\documentclass[11pt,a4paper,reqno]{amsart}
\usepackage[english]{babel}
\usepackage[utf8]{inputenc} 
\usepackage{fancyhdr}
\usepackage{indentfirst}
\usepackage{color}     
\usepackage{graphicx}   
\usepackage{newlfont}

\usepackage{commath}
\usepackage{amsfonts}
\usepackage{amsmath}
\usepackage{amssymb}
\usepackage{amsthm}
\usepackage{latexsym}
\usepackage{mathrsfs}
\usepackage{mathtools}
\mathtoolsset{showonlyrefs=true}
\usepackage{verbatim}
\usepackage[all]{xy}
\usepackage{lineno}
\usepackage{leftidx}
\usepackage[hidelinks]{hyperref}
\usepackage{units}
\usepackage{xfrac}
\usepackage{faktor}
\usepackage{float}
\usepackage{tikz}
\usetikzlibrary{arrows, positioning, automata}

\usepackage{chngcntr}
\usepackage{apptools}

\usepackage[flushleft]{threeparttable}

\theoremstyle{plain} 
\newtheorem*{theo*}{Theorem}
\newtheorem*{cor*}{Corollary}
\newtheorem*{con*}{Conjecture}
\newtheorem{theo}{Theorem}[section] 
\newtheorem{prop}[theo]{Proposition}
\newtheorem{cor}[theo]{Corollary}
\newtheorem{lem}[theo]{Lemma}

\theoremstyle{definition}

\theoremstyle{definition}
\newtheorem{defin}[theo]{Definition}
\newtheorem{ex}[theo]{Example}

\theoremstyle{remark}
\newtheorem{rem}[theo]{Remark}
\newtheorem*{rem*}{Remark}
\newtheorem*{cap*}{Caption}
\newtheorem*{genref}{General References}

\AtAppendix{\counterwithin{theo}{section}}

\newcommand{\Z}{\mathbb{Z}}
\newcommand{\Zplus}{\Z_{>0}}
\newcommand{\Zpluseq}{\Z_{\geq0}}

\newcommand{\C}{\mathbb{C}}

\newcommand{\Uone}{\operatorname{U}(1)}

\newcommand{\Aut}{\operatorname{Aut}}

\newcommand{\Hom}{\operatorname{Hom}}
\newcommand{\Mod}{\operatorname{Mod}}
\newcommand{\Hilb}{\operatorname{Hilb}}
\newcommand{\catVec}{\operatorname{Vec}}
\newcommand{\Rep}{\operatorname{Rep}}
\newcommand{\Repu}{\operatorname{Rep}^{\operatorname{u}}}
\newcommand{\Repf}{\operatorname{Rep}^{\operatorname{f}}}
\newcommand{\opu}{\operatorname{u}}

\newcommand{\J}{\mathcal{J}}
\newcommand{\A}{\mathcal{A}}
\newcommand{\B}{\mathcal{B}}

\newcommand{\parzero}{{\overline{0}}}
\newcommand{\parone}{{\overline{1}}}

\newcommand{\scalar}{(\cdot|\cdot)}
\newcommand{\curlyscalar}{\{\cdot|\cdot\}}

\newcommand{\Tr}{\operatorname{Tr}}

\newcommand{\one}{\mathbf{1}}
\newcommand{\id}{\operatorname{id}}

\def\III{{I\!I\!I}}

\begin{document}
	
\author{Sebastiano Carpi}
	\address{Dipartimento di Matematica, Universit\`a di Roma ``Tor Vergata'', Via della Ricerca Scientifica, 1, 00133 Roma, Italy\\
		E-mail: {\tt carpi@mat.uniroma2.it}
	}
\author{Tiziano Gaudio}
	\address{Department of Mathematics and Statistics, Lancaster University, Lancaster LA1 4YF, UK \\
			E-mail: {\tt t.gaudio2@lancaster.ac.uk}
	}
\author{Luca Giorgetti}
    \address{Dipartimento di Matematica, Universit\`a di Roma ``Tor Vergata'', Via della Ricerca Scientifica, 1, 00133 Roma, Italy\\
    	E-mail: {\tt giorgett@mat.uniroma2.it}
    }
\author{Robin Hillier} 
	\address{Department of Mathematics and Statistics, Lancaster University, Lancaster LA1 4YF, UK\\
		E-mail: {\tt r.hillier@lancaster.ac.uk}
	}
\title[Haploid algebras in $C^*$-tensor categories and the Schellekens list]{Haploid algebras in $C^*$-tensor categories and the Schellekens list
}

\date{11 May 2023}

\maketitle

\begin{center} {\small {\it Dedicated to Roberto Longo on the occasion of his 70th birthday}}
\end{center}
	
\begin{abstract}
We prove that a haploid associative algebra in a $C^*$-tensor category $\mathcal{C}$ is equivalent to a Q-system (a special $C^*$-Frobenius algebra) in $\mathcal{C}$ if and only if it is rigid. This allows us to prove the unitarity of all the 70 strongly rational holomorphic vertex operator algebras with central  charge $c=24$ and non-zero weight-one subspace, corresponding to entries 1--70 of the so called  Schellekens list. Furthermore, using the recent generalized deep hole construction of these vertex operator algebras, we prove that they are also strongly local in the sense of Carpi, Kawahigashi, Longo and Weiner and consequently we obtain some new holomorphic conformal nets associated to the entries of the list. Finally, we completely classify the simple CFT type vertex operator superalgebra extensions of the unitary $N=1$ and $N=2$ super-Virasoro vertex operator superalgebras with central charge $c<\frac{3}{2}$ and $c<3$ respectively, relying on the known classification results for the corresponding superconformal nets.	
\end{abstract}
	

\tableofcontents

\section{Introduction}

The notion of (associative) algebra over a field, namely a vector space with a bilinear associative multiplication, is ubiquitous in mathematics. If the underlying category of vector spaces $\catVec$ is replaced with another \emph{tensor} category $\mathcal{C}$, one can also talk of algebras (or algebra objects) \emph{in} $\mathcal{C}$. 
Namely, an algebra $A = (X,m,\iota)$ is a triple where $X$ is an object in $\mathcal{C}$, and $m:X\otimes X \to X$ and $\iota: \id_{\mathcal{C}} \to X$ are morphisms in $\mathcal{C}$ playing the role of multiplication and unit. Here, $\id_{\mathcal{C}}$ denotes the tensor unit, which is just the base field $\mathbb{F}$ in the case $\mathcal{C} = \catVec$, seen as a vector space (and algebra) over itself. The associativity of  $A$ is naturally expressed in terms of the associativity constraints in $\mathcal{C}$. One can also consider tensor categorical versions of algebras with more structure, such as coalgebras or Frobenius algebras.  In this paper we will pay special attention to haploid algebras (also called irreducible or connected algebras) in $\mathcal{C}$. An algebra  $A = (X,m,\iota)$ in $\mathcal{C}$ is haploid if $\mathrm{dim} \Hom (\id_{\mathcal{C}}, X) = 1$.

In relatively recent years the study of algebras in tensor categories has become of central importance in different areas of mathematics, e.g.\ in the theory of subfactors, conformal field theory (CFT), topological field theories and  quantum groups, see e.g.\  
\cite{BKLR15,FS10,Kaw15,HKL15,KO02,Mue10b,NY18}.

Motivated by the applications to operator algebras and quantum field theory, which are often constructed and studied as families of operators on a complex Hilbert space, 
we shall restrict ourselves to the case $\mathbb{F} = \C$, and we will mainly consider  $C^*$-tensor categories (also known as unitary tensor categories), see e.g.\ \cite{BKLR15,LR97,NT13}.

 If $\mathcal{C}$ is a $C^*$-tensor category, it is  useful to consider algebra objects in $\mathcal{C}$ with a unitary structure, such as $C^*$-Frobenius algebras or Q-systems (special $C^*$-Frobenius algebras), see e.g.\ \cite{BKLR15, Gui21c,Lon94,NY18}.  For example, Q-systems play an important role in the theory of subfactors. Namely, a subfactor with finite Jones index is equivalently well-described by a Q-system in a tensor category of endomorphisms or bimodules of a factor. Haploid Q-systems (also called irreducible or connected Q-systems) correspond to irreducible subfactors.

 A natural question that arises is the following: can we characterize those algebra objects in a $C^*$-tensor category $\mathcal{C}$ that are equivalent to Q-systems? Similar questions have been raised e.g.\ in \cite{GS12}, in \cite[Section 3.1]{BKLR15} and in \cite{CHJP22}. In this paper, we give a complete answer to this question in the haploid case. We prove that a haploid algebra object in a $C^*$-tensor category is equivalent to a Q-system if and only if it is rigid in the sense of \cite{KO02},
 cf.\ Theorem \ref{theo:haploid_rigid_iff_Q-system}. If $A$ is \emph{not} haploid, the conclusion of the previous theorem is false. Namely, the algebra $A$ may be rigid but not equivalent to a Q-system, as we show in Example \ref{ex:CE1}. Moreover, the haploid condition alone does not imply rigidity, as we show in Example \ref{ex:CE2}.  We also prove that if two normalized haploid Q-systems are equivalent as algebra objects, then they are unitarily equivalent (and thus, they also give rise to isomorphic subfactors), cf.\ Theorem \ref{theo:algebras_equiv_as_Qsystems}. The rigidity of $A$ is equivalent to the existence of a Frobenius structure on $A$ by \cite{FRS02}, cf.\ Remark \ref{rem:rigid_iff_Frobenius}. Hence, our results show that there is a one-to-one correspondence between equivalence classes of haploid Frobenius algebras in a $C^*$-tensor category $\mathcal{C}$ and unitary equivalence classes of 
normalized haploid Q-systems in $\mathcal{C}$. As a consequence, we prove that every  semisimple indecomposable module category over a unitary fusion category is equivalent (as a module category) to a $C^*$-module category, answering to a question in \cite{Reu19}.

Our main motivations come from various applications of the above results to chiral CFT and related mathematical structures, such as 
vertex operator algebras (VOAs), subfactors and conformal nets. 

We first briefly describe  here an application concerning  the general relations between subfactors and CFT, see e.g.\ \cite{EK98,Gan06,Kaw15,Kaw19}, from the VOA point of view. Let $\mathcal{C}$ be a unitary fusion category and let $N$ be the hyperfinite type $\III_1$ factor. Then, by  \cite{HY00,Pop95}, see also \cite[Section 2]{Izu17} and \cite[Theorem 3.10]{HP20},  $\mathcal{C}$ can be realized in an essentially unique way as a full subcategory of $\operatorname{End}(N)$ so that we can assume that $\mathcal{C} \subset \operatorname{End}(N)$. Now, let $V$ be a strongly rational VOA so that $\Rep(V)$ is a modular tensor category \cite{Hua08} and thus it is in particular a braided tensor category. By \cite{HKL15,KO02} the simple CFT type extensions of $V$ correspond to rigid haploid commutative algebras in $\Rep(V)$. Hence, if $\Rep(V)$ is tensor equivalent to a unitary fusion category $\mathcal{C} \subset \operatorname{End}(N)$ then, by our result, every simple CFT type extension of $V$ gives rise to a haploid commutative Q-system (also called haploid local Q-system) in $\mathcal{C}$. Then, it also gives rise  to a finite index irreducible subfactor $N \subset M$ which is in fact a braided subfactor in the sense of \cite{BEK99}. In this way one obtains braided subfactors from VOA extensions. 

More direct connections between subfactor theory and CFT come from the theory of conformal nets, see e.g.\ \cite{Kaw15,Kaw19,LR95}. If $\mathcal{A}$ is a completely rational conformal net, then the irreducible (necessarily finite index) local extensions $\mathcal{B} \supset \mathcal{A}$ directly give braided type $\III_1$ subfactors: $\mathcal{B}(I) \supset \mathcal{A}(I)$ for every interval $I\subset S^1$ of the circle. Moreover, such extensions correspond to haploid commutative Q-systems with trivial twist in the modular tensor category $\Repf(\mathcal{A})$ of finite index representations of $\mathcal{A}$.  Now, if $V$ is a strongly rational VOA with $\Rep(V)$ equivalent to  $\Repf(\mathcal{A})$ as a modular tensor category, then, arguing as before, we get a one-to-one correspondence between simple CFT type VOA extensions of $V$ and local irreducible extensions of $\mathcal{A}$, cf.\ Theorem \ref{theo:equiv_VOA_net_extensions}. This allows to transfer classification results for VOA extensions to classification results for conformal net extensions and vice versa. As an example we will give a complete classification of the simple CFT type vertex operator superalgebras (VOSAs) extensions of the unitary $N=1$ and $N=2$ super-Virasoro VOSAs with central charge 
$c< 3/2$ and $c<3$ repectively by transferring to the vertex algebra setting the classification results for superconformal nets in \cite{CKL08,CHKLX15}, cf.\ Theorem \ref{theo:classification_N=1} and Theorem \ref{theo:classification_N=2}.

A typical situation where we have that $\Rep(V)$ is equivalent to  $\Repf(\mathcal{A})$ as a modular tensor category comes from the general relation between unitary VOAs and 
conformal nets introduced and studied in \cite{CKLW18}. In the latter work it was shown that a simple unitary VOA $V$ satisfying a certain analytic condition called strong locality gives rise to a conformal net $\mathcal{A}_V$ in a natural way.  Moreover, many examples of unitary VOAs have been shown to be strongly local. It is conjectured in \cite{CKLW18} that every simple unitary VOA is strongly local and that every conformal net comes from a strongly local VOA. Moreover, it has been conjectured that if $V$ is a unitary strongly rational VOA, then $\mathcal{A}_V$ is completely rational and $\Rep(V)$ is equivalent to $\Repf(\mathcal{A}_V)$ as a modular tensor category, see \cite{Kaw15,Kaw19}. The latter conjectures have been recently proved for a relevant family of unitary strongly rational VOAs, thanks to an impressive progress on the representation theory aspects of the correspondence between VOAs and conformal nets introduced in \cite{CKLW18}, see e.g.\ \cite{Gui20,Ten19a}. 

Our abstract characterization of haploid Q-systems in $C^*$-tensor categories and its consequences on module categories over $C^*$-tensor categories (see below) are sufficiently strong to consider applications beyond the study of local extensions of chiral CFTs. The main point is that we do not assume the haploid algebras to be commutative nor the underlying $C^*$-tensor category  to be braided. This leaves open the possibility to  apply our results e.g.\ in the setting of boundary CFT, where non-commutative haploid algebras appear naturally, see e.g.\ \cite{BKL15,BKLR15,CKL13,FRS02,FS10,KR09,LR04,LR09}.
In this paper we will not further elaborate on these potential applications to boundary CFT. 

We now discuss the application that gave the first motivation to this work as also emphasized in its title.  It concerns the unitarity of VOAs in general and the unitarity of holomorphic (i.e.\  with trivial representation theory) VOAs with central charge $c=24$ and non-zero weight-one subspace (the Schellekens list \cite{Sch93}). 
In the VOA setting, unitarity is a special requirement that turns out to be satisfied in an important but specific class of models. It has naturally appeared  more or less explicitly in VOA theory since its beginning in various ways, such as the existence of a positive definite invariant Hermitian form, the existence of a real form with a positive definite invariant bilinear form, and found various applications besides the correspondence with conformal nets, see e.g.\ \cite{CKLW18,DL14,DL15,FLM88,Mas20,Miy04}. 

The unitarity of a VOA can be proved in various ways. First of all there are direct methods. A VOA can be shown to be unitary by a direct construction of a positive definite invariant Hermitian form, see e.g.\ \cite{CKLW18,CTW22b,DL14,FLM88,Miy04}. In various cases the hardest part is the problem of positivity that can be solved e.g.\ by relying on unitarity results for the underlying infinite dimensional Lie algebras \cite{Kac94,KR87} or non-linear Lie algebras \cite{CTW22b}. These direct methods already produce a remarkable family of unitary VOAs such as those constructed from affine Lie algebras, the Virasoro algebra,  the $\mathcal{W}_3$ algebras, even lattices and the famous Moonshine VOA, see \cite{CKLW18,CTW22,DL14}.

A first indirect method comes from embeddings into unitary VOAs. In the terminology used in \cite{CKLW18} the unitary subalgebras of unitary VOAs are
unitary VOAs \cite[Section 5.4]{CKLW18}. This gives many other examples of unitary VOAs such as orbifold VOAs and coset VOAs.  However, many interesting VOA examples come from taking extensions rather than from taking subalgebras. Indeed, the already mentioned holomorphic VOAs with $c=24$ are of this type (they are extensions of unitary affine VOAs) and thus,  it is desirable to have some easy-to-handle method to prove unitarity for extensions of unitary VOAs. One could be tempted to guess that if an extension $U$ of a unitary VOA $V$ is a unitarizable $V$-module then it is unitary. Unfortunately, such a general result is presently not available. Even for $\mathbb{Z}_n$-simple current extensions by a unitary simple current such general unitarity result  is known only for the case $n=2$ \cite[Theorem 3.3]{DL14}, see also \cite[Corollary 2.7.3]{Gau21} and \cite{CGH}. 

Remarkable progress in the study of unitarity of VOA extensions has recently been made in \cite{Gui21c}. Let $V$ be a unitary strongly rational VOA and assume that $V$ is strongly unitary \cite{Gui21c}, i.e., that every $V$-module is unitarizable or, in other words, that the $C^*$-category $\Repu(V)$ of unitary $V$-modules is linearly equivalent to the modular tensor category $\Rep(V)$.   In \cite{Gui19I,Gui19II} Gui defined certain non-degenerate Hermitian forms on the duals of the spaces of the intertwining operators of $V$ and showed that if these forms are positive definite, then they naturally turn $\Repu(V)$ into a unitary  modular tensor category equivalent to $\Rep(V)$. In particular, $\Repu(V)$ becomes a $C^*$-tensor category tensor equivalent to $\Rep(V)$. 
If $V$ is strongly rational, strongly unitary and satisfies the above positivity condition it is said to be completely unitary \cite{Gui21c}. Many strongly rational unitary VOAs, including rational affine VOAs, lattice VOAs and unitary rational Virasoro VOAs, have been shown to be completely unitary, see e.g.\ 
\cite{Gui19II,Gui20,Gui21a,Ten19a}. It follows directly from the results in \cite{Gui21c} that if $U$ is a simple CFT type VOA extension of a completely unitary 
VOA $V$ then $U$ is unitary (and in fact completely unitary) if and only if the corresponding commutative haploid algebra in $\Repu(V)$ is equivalent to a Q-system. Hence, since rigidity of the latter haploid algebra follows from the simplicity of $U$ \cite{KO02}, it follows from our characterization of haploid Q-systems in $C^*$-tensor categories that every simple CFT type extension of a completely unitary VOA is unitary, and in fact completely unitary, cf.\ Theorem \ref{theo:unitarity_voa_extensions}. 
In particular, every simple CFT type extension of a unitary rational affine VOA is unitary. 

As already mentioned, one of our motivations comes from the unitarity problem for the holomorphic VOAs with $c=24$. Holomorphic VOAs play a central role in vertex algebra theory. The famous Moonshine VOA \cite{FLM88}, which was one of the motivations for developing vertex algebra theory \cite{Bor86} and which was crucial in the proof of the Moonshine conjecture \cite{Bor92}, is a holomorphic VOA with $c=24$. Moreover, holomorphic VOAs appear to be deeply related with modular forms \cite{DM00} and with the geometry of moduli spaces of complex curves \cite{Cod20}. They are expected to correspond to holomorphic conformal nets \cite{KL06,KS14}. The latter play an important role in the search for a CFT realization of exotic subfactors such as the Haagerup subfactor
\cite{EG11,EG22,Bis16,Bis17}. 

The central charge $c$ of a holomorphic VOA must be a positive integer multiple of 8. For  $c=8$ and $c=16$ the holomorphic VOAs were classified in  \cite{DM04a} and turned out to be exactly the lattice VOAs associated to the self-dual even positive definite lattices of rank 8 and 16, respectively.   The classification 
of holomorphic VOAs $V$ with $c=24$ and weight-one subspace $V_1\neq \{ 0\}$ began with the seminal work of Schellekens \cite{Sch93} and has recently been completed thanks to the work of various authors, see e.g.\ \cite{EMS20,ELMS21,LS19,MS20}. There are exactly 70 such VOAs  corresponding to entries 1--70 of the Schellekens list \cite{Sch93}. The Moonshine VOA is a holomorphic VOA with $V_1= \{0 \}$ and corresponds to entry 0 in the list. It has been conjectured in \cite{FLM88} that the Moonshine VOA is the \emph{only} holomorphic VOA with $c=24$ and $V_1= \{0 \}$. A proof of this conjecture would complete the classification by showing that there are exactly 71 holomorphic VOAs with $c=24$ and that they are in one-to-one correspondence with entries 0--70 in \cite{Sch93}.  

Let us now com back to the unitarity problem. The Moonshine VOA is known to be unitary \cite{DL14,CKLW18}. Entry 1 in \cite{Sch93} corresponds to the Leech lattice VOA, which is also known to be unitary. The remaining entries 2--70 are simple CFT type extensions of rational affine VOAs, the subVOAs generated by the $V_1$ subspaces, which are semisimple Lie algebras for these entries. Hence, by our results they are all unitary, cf.\  Theorem \ref{theo:unitarity_holomorphicVOAs}.
 Some of them were already known to be unitary. For example, 24 of the 71 entries correspond to  the lattice VOAs associated to Niemeier lattices. On the other hand, unitarity was not known in various cases.  

Once unitarity is proved, it is natural to ask if these holomorphic VOAs are also strongly local and generate holomorphic conformal nets. Using the recent generalized deep hole construction \cite{MS20} of the holomorphic VOAs $V$ with $c=24$ and $V_1  \neq  \{0\}$ and the results on energy bounds for VOA extensions in \cite{CT}, we are able to prove strong locality for entries 2--70 in \cite{Sch93} 
cf.\ Theorem \ref{theo:strong_locality_holomorphic_VOAs} (the strong locality for entries 0--1 was already known). In this way, we get a family of 71 holomorphic conformal nets with $c=24$ and show that the holomorphic conformal nets already constructed by different methods \cite{KL06,KS14,Xu09,Xu18} are a (proper) subset of this family. From these holomorphic conformal nets, one can obtain many interesting finite index subfactors arguing e.g.\ as in \cite[Section 4]{Xu18}.

A different proof of the unitarity of strongly rational holomorphic VOAs with c=24 and non-zero weight-one subspaces has been found independently  by Lam in \cite{Lam22}.

This paper is organized as follows. In Section \ref{sec:cat_prelim}, we review the necessary notions of $C^*$-tensor category theory. In particular, the notion of algebra, Frobenius algebra, $C^*$-Frobenius algebra and Q-system. We recall or prove general results to be used in the subsequent sections. We also recall the definition of $A$-module, unitary $A$-module and local (or dyslectic) $A$-module, the latter in the case of unitary braided tensor categories. Note that we do not necessarily restrict ourselves to rigid, nor semisimple, unitary tensor categories, but we do assume the tensor unit to be simple.

In Section \ref{sec:main_theo}, we completely settle the problem of understanding when a (Frobenius) algebra in a unitary tensor category can be \emph{unitarized}, namely when it is equivalent to a Q-system, and we discuss some first applications of this characterization. 

In Section \ref{sec:applications}, we give various general chiral CFT applications  with some first examples.

In Section \ref{sec:unitarity_schellekens_list}, we consider applications to the holomorphic VOAs/conformal nets with $c=24$ (the Schellekens list). 

In Section \ref{sec:superconformal_VOAs}, we discuss our classification results  for $N=1$ superconformal VOSAs with $c<3/2$ in the unitary discrete series  and for $N=2$ superconformal VOSAs with $c<3$ in the unitary discrete series. 

In the appendix, we prove a ``unitarization'' result for the compact automorphism groups of unitary VOAs which is needed in our proof of strong locality of holomorphic VOAs with $c=24$. In the course of the proof, we show various properties of the automorphism groups of unitary VOAs which appear to be of independent interest and which can be seen as a complement to \cite[Section 5.3]{CKLW18}.

\section{Preliminaries on tensor categories}	\label{sec:cat_prelim}

Let $\mathcal{C}$ be a \textbf{$C^*$-tensor category} \cite{GLR85,DR89,LR97}, not necessarily rigid, nor semisimple, but with simple (i.e., irreducible) tensor unit $\id_\mathcal{C}$. In particular, there is an anti-linear map 
\begin{equation}
\Hom(X,Y) \ni S \mapsto S^* \in \Hom(Y,X)
\end{equation}
such that $S^{**}=S$ for all morphisms $S$ in $\mathcal{C}$. Here $\Hom(X,Y)$ denotes the space of morphisms $X \to Y$ for the objects $X,Y \in \mathcal{C}$.  The morphism spaces are complex Banach spaces and their norm satisfies the $C^*$-identity
$\|S^*S\| = \|S\|^2$. Moreover, we assume that, for all $X,Y \in \mathcal{C}$ and all $S \in \Hom(X,Y)$, $S^*S$ is a positive element in the unital $C^*$-algebra 
$\Hom(X,X)$.   

Moreover, in order to accommodate for the applications of our main result to conformal nets and vertex operator algebras in a more uniform way,  and in spite of  
Mac Lane's coherence theorem \cite{Mac98},
we do not consider only \emph{strict} tensor categories. See \cite[Definition 2.1.1]{NT13}. 
Given $X,Y,Z$ objects in $\mathcal{C}$, we denote by
\begin{equation}
	a_{X,Y,Z}: X\otimes (Y\otimes Z)\longrightarrow (X\otimes Y)\otimes Z
\end{equation}
and
\begin{equation}
	l_X: \id_\mathcal{C}\otimes X\longrightarrow X, \qquad r_X:X\otimes \id_\mathcal{C}\longrightarrow X
\end{equation}
respectively the associator and the left/right unitors in $\mathcal{C}$, which make the usual pentagon and triangle diagrams commute and which are assumed to be unitary in the $C^*$ context. 
We refer to \cite{Mue10b,EGNO15,BKLR15,NY18,BCEGP20} for general background. 

Recall the following definitions:

\begin{defin}
	An object $X$ in $\mathcal{C}$ is said to be \textbf{rigid} (or dualizable) if there is an object $\overline{X}$ (called conjugate or dual) and a pair of morphisms $e_X : \overline{X} \otimes X \to \id_\mathcal{C}$ (called evaluation), $i_X : \id_\mathcal{C} \to X \otimes \overline{X}$ (called coevaluation), fulfilling the \emph{conjugate equations} (also called rigidity or duality equations):
	\begin{equation} \label{eq:conj_eqns}
		\begin{split}
			r_X(\one_X\otimes e_X)a_{X,\overline{X},X}^{-1}(i_X\otimes \one_X) l_X^{-1} &= \one_X \\
			l_{\overline{X}}(e_X\otimes \one_{\overline{X}})a_{{\overline{X}},X,{\overline{X}}}(\one_{\overline{X}}\otimes i_X) r_{\overline{X}}^{-1}
			&= \one_{\overline{X}}
		\end{split}
	\end{equation}
	where $\one_X$ denotes the identity morphism in $\Hom(X,X)$. The category $\mathcal{C}$ itself is called \emph{rigid} 
	if every object $X$ in $\mathcal{C}$ is rigid.  For a rigid object $X$ in $\mathcal{C}$, denote by $d(X)$ the 
intrinsic/$C^*$-tensor categorical \emph{dimension} of $X$. Namely, the 
positive number $d(X)$ defined by  $d(X) \one_{\id_\mathcal{C}}  = e_X e_X^* = i_X^* i_X$, where $e_X$ and $i_X$ are a \emph{standard solution} 
of the conjugate equations \eqref{eq:conj_eqns}, see \cite{LR97}, cf.\ \cite{GL19}.	
\end{defin}

\begin{rem}
	The equations in \eqref{eq:conj_eqns} are usually taken as a definition of \emph{left} duals. One needs another pair of morphisms $e'_X : X \otimes \overline{X} \to \id_\mathcal{C}$, $i'_X : \id_\mathcal{C} \to \overline{X} \otimes X$, fulfilling two equations analogous to \eqref{eq:conj_eqns}, in order to define \emph{right} duals. In the $C^*$, or even *, context, duals are automatically both left and right when they exist, as one can take $e'_X = i_X^*$ and $i'_X = e_X^*$. 
\end{rem}

\begin{defin}
	An \textbf{algebra} in $\mathcal{C}$ is a triple $A=(X, m, \iota)$, where $X$ is an object in $\mathcal{C}$, $m:X\otimes X \to X$ is a morphism (called multiplication), and $\iota: \id_\mathcal{C} \to X$ is a monomorphism (called unit), fulfilling the associativity and the unit property: 
	\begin{equation}
		\begin{split}
			m (m \otimes \one_X) a_{X,X,X} &= m (\one_X \otimes m) \\
			m (\iota \otimes \one_X) l_X^{-1} &= \one_X \! = m (\one_X \otimes \iota) r_X^{-1} \,.
		\end{split}
	\end{equation}
\end{defin}

\begin{rem}
	We shall consider only associative algebras. Thanks to the * structure on $\mathcal{C}$, every algebra in $\mathcal{C}$ is also a (coassociative) \emph{coalgebra} with comultiplication $m^* : X \to X \otimes X$ and counit $\iota^* : X \to \id_{\mathcal{C}}$.
\end{rem}

\begin{defin}
	A \textbf{$C^*$-Frobenius algebra} in $\mathcal{C}$ is an algebra $A=(X, m, \iota)$ such that $m$ and $m^*$ fulfill the Frobenius relations:
	\begin{equation}
		(m \otimes \one_X) a_{X,X,X} (\one_X \otimes m^*) = m^*m =  (\one_X \otimes m) a_{X,X,X}^{-1} (m^* \otimes \one_X).
	\end{equation}
\end{defin}

\begin{rem}
	Note that either of the two equalities above implies the other one, by taking the $^*$, as the associator is unitary. Obviously, $C^*$-Frobenius algebras are special cases of \emph{Frobenius algebras}, see \cite{FRS02,FS03,Mue03,Yam04a}, where the comultiplication $\Delta:X\to X\otimes X$ is not required to be equal to $m^*$.
\end{rem}

Recall that the tensor unit $\id_\mathcal{C}$ is simple, namely $\Hom(\id_\mathcal{C},\id_\mathcal{C}) \cong \C$, i.e., $\Hom(\id_\mathcal{C},\id_\mathcal{C}) = \C \one_{\id_\mathcal{C}}$.

\begin{defin}
	An algebra $A=(X, m, \iota)$ in $\mathcal{C}$ is called \textbf{normalized} if $\iota^* \iota = \one_{\id_\mathcal{C}}$ (namely the unit $\iota$ is an isometry, hence $\iota \iota^*$ is a self-adjoint projection in $\Hom(X,X)$). Furthermore, $A$ is called \textbf{special} (or sometimes also strongly separable) if $m m^*$ is a scalar multiple of $\one_X$, where the scalar is necessarily a positive real number, in $\Hom(X,X)$.
\end{defin}

\begin{defin}  \label{def:equiv_and_unitary_equiv}
	Two algebras $A=(X, m, \iota)$ and $B=(Y, n, j)$ in $\mathcal{C}$ are called \textbf{equivalent} if there is an isomorphism $S\in\Hom(X,Y)$ which intertwines the two algebra structures: 
	\begin{equation}
		\begin{split}
			S m &= n (S \otimes S) \\
			S \iota &= j.
		\end{split}
	\end{equation}
	The algebras $A$ and $B$ are called \textbf{unitarily equivalent} if $S$ is in addition unitary.
\end{defin}

\begin{rem}   \label{rem:associative_C-algebra_equiv_normalized}
	Note that every algebra $A=(X, m, \iota)$ in $\mathcal{C}$ is equivalent to a normalized one. Indeed, $\iota^*\iota=\alpha \one_{\id_\mathcal{C}}$ (as $\Hom(\id_\mathcal{C},\id_\mathcal{C}) \cong \C$) for some positive real number $\alpha$ and thus $A_S := (X, m_S, \iota_S)$ with $S := \sqrt{\alpha} \one_X$, $m_S:=S^{-1}m(S\otimes S)$ and $\iota_S:=S^{-1}\iota$, is a normalized algebra and $S^{-1}$ is an equivalence from $A$ to $A_S$.
\end{rem}

\begin{lem}  \label{lem:mu_strictly_positive}
	Let $A=(X, m,\iota)$ be an algebra in $\mathcal{C}$. Then, $mm^*$ is a strictly positive element in $\Hom(X,X)$. If, in addition, $A$ is a $C^*$-Frobenius algebra, then it is equivalent to a special $C^*$-Frobenius algebra.
\end{lem}	

\begin{proof}
	The proof of the first statement is \cite[Eq.\ (3.1.6)]{BKLR15} for strict $C^*$-tensor categories. Yet, it clearly works for the general case too, we recall it below for the reader's convenience. $\iota^*\iota=\alpha \one_{\id_\mathcal{C}}$ for some positive real number $\alpha$ and thus $\alpha^{-1}\iota\iota^*$ is a projection in $\Hom(X,X)$, i.e., $\alpha^{-1}\iota\iota^*\leq \one_X$. It follows that
	\begin{equation}
		\begin{split}
			mm^* 
			&=  
			m(\one_X\otimes \one_X) m^* \\
			&\geq 
			\alpha^{-1} m(\one_X\otimes \iota\iota^*) m^* \\
			&=
			\alpha^{-1} m(\one_X\otimes \iota)r_X^{-1}[ m(\one_X\otimes \iota)r_X^{-1}]^*\\
			&= 
			\alpha^{-1} \one_X	
		\end{split}
	\end{equation}
	where we have used the unitarity of $r_X$ for the second equality and the unit property of $A$ for the last one.
	The second statement follows as in the second part of the proof of \cite[Lemma 3.5]{BKLR15} and in the proof of \cite[Corollary 3.6]{BKLR15}, by the associativity and the Frobenius property and by taking the non-trivial associator into account.
\end{proof}	

The following and somehow converse statement, originally due to \cite{LR97}, holds:

\begin{lem}
	If $A$ is a special algebra in $\mathcal{C}$, then $A$ is $C^*$-Frobenius.
\end{lem}

\begin{proof}
	The proof can be adapted from the one in the strict case given in \cite[Lemma 3.7]{BKLR15}.
\end{proof}

The following \emph{haploid} condition is motivated by the applications to extensions of chiral conformal field theories (CFTs). It is the abstract description of the irreducibility or simplicity of an extension (uniqueness of the vacuum), both in the conformal net and in the vertex operator algebra formalism. Finite index and conformal extensions of chiral CFTs are necessarily irreducible, hence the haploid case is general for our applications.

\begin{defin}
	An algebra $A=(X, m, \iota)$ in $\mathcal{C}$ is called \textbf{haploid} (or connected, or sometimes also irreducible) if $\Hom(\id_\mathcal{C},X) \cong \C$, i.e., $\Hom(\id_\mathcal{C},X) = \C \iota$.
\end{defin}

In the haploid case, special algebras are the same as $C^*$-Frobenius algebras, without the need of passing to equivalent algebras:

\begin{lem}  \label{lem:frobenius_iff_special}
	A haploid algebra $A$ in $\mathcal{C}$ is special if and only if it is $C^*$-Frobenius.
\end{lem}

\begin{proof}
	The \lq\lq if\rq\rq part follows as in the proof of \cite[Lemma 3.3]{BKLR15}.
\end{proof}

If $A$ is haploid, then $\iota$ determines a unique morphism $\varepsilon : X \to \id_\mathcal{C}$ such that $\varepsilon \iota = \one_{\id_\mathcal{C}}$.
In particular, if $A$ is in addition normalized, then $\varepsilon = \iota^*$. Following \cite{KO02}:

\begin{defin}  \label{def:rigid_haploid_alg}
	A haploid algebra $A=(X, m, \iota)$ in $\mathcal{C}$ is called \textbf{rigid} if the morphism $e_A := \varepsilon m : X\otimes X \to \id_\mathcal{C}$ admits a coevaluation $i_A: \id_\mathcal{C} \to X\otimes X$, fulfilling the conjugate equations \eqref{eq:conj_eqns}. In particular, the object $X$ is rigid with conjugate $\overline{X} = X$.
\end{defin}

\begin{rem}  \label{rem:CstarFrob_is_rigid_and_real}
	More generally, if $A=(X, m, \iota)$ is a (not necessarily haploid) $C^*$-Frobenius algebra in $\mathcal{C}$, then $e_A := \iota^* m : X\otimes X \to \id_\mathcal{C}$ and $i_A := m^* \iota: \id_\mathcal{C} \to X\otimes X$, thus $e_A = i_A^*$, provide a solution of the conjugate equations \eqref{eq:conj_eqns} for $X$ and $\overline{X} := X$.
\end{rem}

In the general (not necessarily haploid) case: 

\begin{defin}
	We call \textbf{Q-system} in $\mathcal{C}$ a special $C^*$-Frobenius algebra $A$ in $\mathcal{C}$.
\end{defin}	

\begin{defin}
	A Q-system is \textbf{normalized} or \textbf{haploid}, if it is respectively normalized or haploid as an algebra. A Q-system is \textbf{standard} if $e_A = \iota^* m$ and $i_A = m^* \iota = e_A^*$ are a \emph{standard solution} of the conjugate equations \eqref{eq:conj_eqns} for $X$ and $\overline{X} = X$, see \cite{LR97}, cf.\ \cite{GL19}.
\end{defin}

\begin{rem}  \label{rem:std_Qsys}
	If the Q-system is normalized, then it is standard if and only if the scalar $\lambda$ in $m m^* = \lambda \one_X$ equals the dimension $d(X)$ of $X$ seen as rigid object in $\mathcal{C}$. 
	Namely, $\lambda \one_{\id_{\mathcal{C}}}= d(X) \one_{\id_{\mathcal{C}}}
	= \iota^* m m^* \iota $, if $e_A = \iota^* m = i_A^*$ are standard.	
	If the Q-system is haploid, then it is \emph{automatically} standard, see \cite[Section 6]{LR97}, see also \cite[Remark 5.6(3)]{Mue03}, \cite[Theorem 2.9]{NY18}, \cite[Proposition 2.20]{Gui21c}.
\end{rem}	

In our applications to conformal nets and vertex operator algebras in Section \ref{sec:applications}, but not in our main result in Section \ref{sec:main_theo}, the category $\mathcal{C}$ is in addition unitarily braided.
Let $\mathcal{C}$ be a \textbf{braided $C^*$-tensor category}. Denote by
\begin{equation}
	b_{X,Y}: X\otimes Y \longrightarrow Y\otimes X
\end{equation}
the \emph{braiding} in $\mathcal{C}$, which makes the usual hexagonal diagrams commute and which we will assume to be unitary in the $C^*$ context. For rigid objects $X$ in $\mathcal{C}$, together with the dimension $d(X)$ recalled above, denote by $\omega(X)$ the canonical unitary \emph{twist} of $X$ in $\mathcal{C}$. Namely, 
\begin{equation} \label{eq.:twist}
\omega(X) := l_X 
(e_X \otimes \one_X) a_{\overline{X},X,X} (\one_{\overline{X}} \otimes 
b_{X,X}) a_{\overline{X},X,X}^{-1} (e_X^* \otimes \one_X) l_X^{-1} : X \to X
\end{equation}
if $e_X$ and $i_X$ are standard, see \cite{LR97} and \cite{Mue00}. 

\begin{rem}\label{rem:braiding}
	In the special case of unitary \emph{fusion} categories, which cover the applications to \emph{rational} chiral CFTs, every braiding is automatically unitary thanks to \cite[Theorem 3.2]{Gal14}. Moreover, every unitary braided fusion category admits a unique ribbon structure making it into a unitary pre-modular tensor category, namely the one defined by the canonical twist $\omega$, see 
\cite[Theorem 3.5]{Gal14}.
\end{rem}

\begin{rem}\label{rem:twistinverible} Let $\mathcal{C}$ be a unitary fusion category. An object $X$ in $\mathcal{C}$ is said to be invertible (or a \emph{simple current} in the CFT context) if  $X \otimes \overline{X}$ and $\overline{X} \otimes X$ are isomorphic to 
$\id_\mathcal{C}$. Equivalently, $X$ is invertible if $d(X) = 1$. If $X$ is invertible,  then $X$ and $X \otimes X$ are simple. As a consequence, if $\mathcal{C}$ is a unitary pre-modular tensor category, there are complex numbers $\alpha(X)$ and $\beta(X)$ such that  $\omega(X) = \alpha(X)\one_X$ and $b_{X,X} = \beta(X) \one_{X \otimes X}$. Hence, from Eq. \eqref{eq.:twist} we find 
\begin{equation} 
	\begin{split}
\alpha(X) \one_X &=   \beta(X)l_X 
(e_X \otimes \one_X) a_{\overline{X},X,X} (\one_{\overline{X}} \otimes 
\one_{X \otimes X}) a_{\overline{X},X,X}^{-1} (e_X^* \otimes \one_X) l_X^{-1} \\
&= \beta(X) d(X) \one_X = \beta(X) \one_X
	\end{split}
\end{equation}
so that $\alpha(X) = \beta(X)$ for every invertible object $X$ in $\mathcal{C}$, cf.\ \cite{Reh90}. If the simple current $X$ is not isomorphic to 
$\id_{C}$ but  $X\otimes X$ is (equivalently $X \cong \overline{X}$), then $X$ is called a \emph{$\mathbb{Z}_2$-simple current}. 
\emph{$\mathbb{Z}_n$-simple currents} are defined analogously.      
\end{rem}

\begin{defin}
	An \textbf{algebra} $A=(X, m,\iota)$ in $\mathcal{C}$ is called \textbf{commutative}, with respect to the given braiding, if the following commutativity condition holds:
	\begin{equation} \label{eq:comm_alg}
		m b_{X,X} = m
	\end{equation}
	or equivalently $m b_{X,X}^\textrm{op} = m$, where $b_{X,X}^\textrm{op} := b_{X,X}^{-1}$ is the opposite self-braiding of $X$.
\end{defin}

We state and sketch the proof of the following fact, which is presumably known to experts, but which we could not find in the literature:

\begin{prop}  \label{prop:trivial_twist}
	Let $A=(X, m,\iota)$ be a standard Q-system in a braided $C^*$-tensor category $\mathcal{C}$.
	If $A$ is commutative, then it has trivial twist, namely $\omega(X) = \one_X$. 
\end{prop}

\begin{proof}
	The unitarity of the canonical twist implies that $e_A (\omega(X) \otimes \one_X) = e_A b_{X,X}^{-1}$, cf.\ \cite[Lemma 4.3]{LR97}. By choosing $e_A = \iota^* m = i_A^*$ as a standard solution of the conjugate equations \eqref{eq:conj_eqns} and by the commutativity condition \eqref{eq:comm_alg}, it follows that $e_A (\omega(X) \otimes \one_X) = e_A$. Thus $\omega(X) = \one_X$ by using the conjugate equations.
\end{proof}

\begin{rem}
	If $A$ is haploid, the conclusion of the previous proposition follows by realizing the rigid braided $C^*$-tensor category generated by $X$ as endomorphisms of a type $I\!I\!I$ factor, see \cite{Yam03,BHP12,GY19}, and by considering the irreducible finite index subfactor associated to $X$, see \cite{Lon94,Jon83}.
	Then $\omega(X) = \one_X$ follows by \cite[Proposition 4]{Reh94}, or by \cite[Corollary 2.20]{BDG21}.
\end{rem}

Recall Definition \ref{def:equiv_and_unitary_equiv}. We shall need the following observations:

\begin{rem}  \label{rem:equiv_and_unitary_equiv}
	The properties of being haploid or rigid or Frobenius, or commutative in the case of braided categories, are invariant under equivalence of algebras. 
	In contrast, the properties of being normalized or special or $C^*$-Frobenius, hence Q-system, or standard Q-system, are invariant under \emph{unitary} equivalence of algebras. 
\end{rem}

\begin{prop}   \label{prop:1-1_equiv_algebras_under_categ_equiv}
	Let $\mathcal{C}$ and $\mathcal{D}$ be two (braided) $C^*$-tensor categories with simple tensor unit. Let $F:\mathcal{C}\to\mathcal{D}$ be a (braided) unitary, i.e., *-preserving, tensor equivalence, with unitary tensorator $f_2: F(X) \otimes F(Y) \to F(X\otimes Y)$, for $X,Y$ objects in $\mathcal{C}$, 
	and unitor $f_0: \id_{\mathcal{D}} \to F(\id_{\mathcal{C}})$, and with inverse $G:\mathcal{D}\to\mathcal{C}$ and unitary tensor natural transformations $G \circ F \Rightarrow I_{\mathcal{C}}$ and $F \circ G \Rightarrow I_{\mathcal{D}}$.
	Then, there is a one-to-one correspondence between equivalence classes of (commutative) algebras in $\mathcal{C}$ and in $\mathcal{D}$ given by 
	$$A=(X,m,\iota) \mapsto F(A):=(F(X), F(m)f_2, F(\iota)f_0).$$
	Moreover, the map $A\mapsto F(A)$ also gives a one-to-one correspondence between equivalence classes of (commutative) Q-systems in 
	$\mathcal{C}$ and in $\mathcal{D}$. 	
\end{prop}

For later use, we recall the definition of module over an algebra in a ($C^*$-)tensor category, the definition of unitary module, and of local module.  

\begin{defin}
		Let $\mathcal{C}$ be a tensor category with simple tensor unit, and $A=(X, m, \iota)$ an algebra in $\mathcal{C}$. A \textbf{left $A$-module} in $\mathcal{C}$ (the definition of right $A$-module in $\mathcal{C}$ is analogous) is a pair $M=(Y, m_Y)$, where $Y$ is an object in $\mathcal{C}$ and $m_Y:X\otimes Y \to Y$ is a morphism (called left action), fulfilling the representation and unit properties: 
		\begin{equation}
			\begin{split}
				m_Y (\one_X \otimes m_Y) a_{X,X,Y}^{-1} &= m_Y (m \otimes \one_Y) \\
				m_Y (\iota \otimes \one_Y) l_Y^{-1} &= \one_Y \,.
			\end{split}
		\end{equation}
		We denote by $\Mod_\mathcal{C}(A)$ the category of left $A$-modules in $\mathcal{C}$.
\end{defin}

Now, let $\mathcal{C}$ be a $C^*$-tensor category with simple tensor unit and $A$ be a Q-system (i.e., a special $C^*$-Frobenius algebra) in $\mathcal{C}$.

\begin{defin}
		A left $A$-module is called \textbf{unitary} (or special, or also standard) if 
		\begin{equation}
			(m \otimes \one_Y) a_{X,X,Y} (\one_X \otimes m_Y^*)
			= m_Y^*m_Y 
			=  (\one_X \otimes m_Y) a_{X,X,Y}^{-1} (m^* \otimes \one_Y) 
		\end{equation}
		or equivalently, see \cite[Lemma 3.23]{BKLR15}, \cite[Remark 2.7(iii)]{NY18}, if $m_Y m_Y^*$ is a scalar multiple of $\one_Y$ in $\Hom(Y,Y)$. In this case, it is the same scalar as $mm^*$ in $\Hom(X,X)$.
		We denote by $\Mod^{\opu}_\mathcal{C}(A)$ the full subcategory of unitary left $A$-modules in $\mathcal{C}$.
\end{defin}

\begin{defin}  \label{def:Modu0(A)}
		Let $\mathcal{C}$ be in addition braided, and $A=(X, m, \iota)$ in addition commutative with respect to the given braiding $b$. 
		A left $A$-module $M=(Y, m_Y)$ in $\mathcal{C}$ is called \textbf{local} (or dyslectic, or also single-valued) if $m_Y b_{Y,X} b_{X,Y} = m_Y$.
		We denote respectively by $\Mod^0_\mathcal{C}(A)$ and $\Mod^{\opu,0}_\mathcal{C}(A)$, the full subcategory of local and unitary local left $A$-modules in $\mathcal{C}$.
	\end{defin}

\section{From algebra objects to Q-systems}	\label{sec:main_theo}

In this section, we show that a haploid algebra in a $C^*$-tensor category (with simple tensor unit) is \emph{equivalent} to a Q-system if and only if it is rigid. Furthermore, we show that the haploid condition (which is general for the sake of studying finite index extensions of conformal nets and vertex operator algebras) cannot be removed from the statement.

Recall the following Perron--Frobenius type result. As customary, we call 
an element $T$ in a unital $C^*$-algebra $\mathfrak{A}$ \emph{positive}, denoted by $T\geq 0$, if it is of the form $T=S^*S$ for some $S \in \mathfrak{A}$.
We call $T$ \emph{strictly positive}, denoted by $T>0$, if $T$ is positive and invertible in $\mathfrak{A}$. 
Differently, a real number $\lambda>0$ is called positive, whereas $\lambda\geq0$ is called non-negative.
Recall also that a linear map $\Phi:\mathfrak{A}\to\mathfrak{A}$ is \emph{positive} if $\Phi(T) \geq 0$ for every $T\geq 0$. It is called \emph{strictly positive} if $\Phi(T) > 0$ for every $T\geq 0$, $T\neq 0$. For a finite dimensional $C^*$-algebra realized on a Hilbert space of dimension $n$, see \cite[Lemma 2.1]{EH78}, we call a positive map $\Phi$ \emph{irreducible} if $(\id_{\mathfrak{A}}+\Phi)^{n-1}$ is strictly positive, where $\id_{\mathfrak{A}}$ is the identity map on $\mathfrak{A}$.

\begin{prop}   \label{prop:perron_frobenius}
	Let $\Phi:\mathfrak{A}\to\mathfrak{A}$ be an irreducible positive map on a finite dimensional $C^*$-algebra $\mathfrak{A}$. Then, there exist a strictly positive element $T$ of $\mathfrak{A}$ and a positive number $\lambda$ such that $\Phi(T) = \lambda T$. Furthermore, if $\Phi(S) = \alpha S$ for some non-zero positive element $S$ in $\mathfrak{A}$ and some complex number $\alpha$, then $\alpha = \lambda$ and $S$ is a scalar multiple of $T$.
\end{prop}	

\begin{proof}
	See \cite[Theorem 2.3 and Theorem 2.4]{EH78}.
\end{proof}

\begin{theo}   \label{theo:haploid_rigid_iff_Q-system}
	Let $\mathcal{C}$ be a $C^*$-tensor category (with simple tensor unit) and $A=(X, m,\iota)$ be a haploid algebra in $\mathcal{C}$. Then, $A$ is rigid if and only if $A$ is equivalent to a Q-system in $\mathcal{C}$. More specifically, $A$ is rigid if and only if $A$ is equivalent to a Q-system $A_S:=(X,S^{-1}m(S\otimes S), S^{-1}\iota)$ in $\mathcal{C}$ for a strictly positive isomorphism $S\in\Hom(X,X)$. Moreover, $A_S$ can always be taken to be normalized, in which case $A$ is normalized if and only if $S^{-1} \iota = \iota$. In either case, the equivalent Q-system $A_S$ is also standard. 
\end{theo}

To render the proof of Theorem \ref{theo:haploid_rigid_iff_Q-system} more transparent, we recall and use a second multiplication operation on the finite dimensional $C^*$-algebra $\Hom(X,X)$, given by the algebra (and coalgebra) structure on $X$. Let $T,S \in \Hom(X,X)$. The \textbf{convolution} of $T$ and $S$ (sometimes also called coproduct, or horizontal 2-box multiplication) is defined by
\begin{equation}  \label{eq:convolution}
	T \ast S := m (T \otimes S) m^* \in \Hom(X,X).
\end{equation}
The terminology comes from subfactor/planar algebra theory \cite{Jon99,BJ00}, and from the subfactor theoretical Fourier transform \cite{Ocn88}, see also \cite{NW95,Bis97,BKLR15,JLW16,JJLRW20,BDG22a}.

The convolution is clearly \textit{bilinear}. By the associativity (and coassociativity) of $A$ and by the naturality of the associator, the convolution is \textit{associative}, i.e., $(T \ast S) \ast R = T \ast (S \ast R)$ for all $T, S, R\in \Hom(X,X)$. Simarly, thanks to the unit property of $A$ and to the naturality of the unitors, it has an \textit{identity element} given by $\iota \iota^* \in \Hom(X, X)$, i.e., $\iota \iota^* \ast T= T = T\ast \iota \iota^*$ for all $T\in \Hom(X,X)$. Moreover, it is \textit{*-preserving} by the unitarity of the associator and of the unitors, i.e., $(T\ast S)^*=T^*\ast S^*$ for all $T, S\in \Hom(X, X)$.

\begin{proof}[Proof of Theorem \ref{theo:haploid_rigid_iff_Q-system}]
To prove the first claim, assume that $A=(X, m,\iota)$ is haploid and rigid and consider the finite dimensional $C^*$-algebra $\mathfrak{A}:=\Hom(X,X)$. Thanks to Lemma \ref{lem:frobenius_iff_special}, it is sufficient to deform $A$ into an equivalent (by Remark \ref{rem:equiv_and_unitary_equiv}, necessarily haploid) algebra satisfying specialness. More specifically, we shall keep $X$ fixed and deform $m$ and $\iota$. Namely, we want to find an invertible element $S\in\mathfrak{A}$ such that $ m_S:=S^{-1}  m (S\otimes S)$ satisfies 
	\begin{equation}  \label{eq:specialness_mu_S}
		m_S( m_S)^*=S^{-1}  m (SS^*\otimes SS^*) m^*(S^{-1})^*=\gamma \one_X
	\end{equation}
for some $\gamma\in\C$, which is necessarily positive. Let $T:=SS^*$ and recall the notation from \eqref{eq:convolution}. Then, finding $S$ as in \eqref{eq:specialness_mu_S} is equivalent to finding a strictly positive element $T\in\mathfrak{A}$ such that 
	\begin{equation}   \label{eq:specialness_strictly_positive}
		T \ast T = \gamma T
	\end{equation}
for some positive real number $\gamma$. 
	
To this end, define a linear map $\Phi:\mathfrak{A} \to \mathfrak{A}$ by setting
	\begin{equation}  \label{eq:map_Phi_def}
		\Phi(T):= T \ast \one_X \qquad\forall T\in\mathfrak{A} \,.
	\end{equation}
We first show that $\Phi$ is an irreducible positive map.
By the finite dimensionality of $\mathfrak{A}$, entailed by rigidity and by \cite[Lemma 3.2]{LR97}, we have that $\mathfrak{A}=\bigoplus_{i=0}^N\mathfrak{A}_i$ for some $N\in\Zplus$, where $\mathfrak{A}_i\cong \operatorname{Mat}_{n_i}(\C)$ for some $n_i\in\Zplus$, for every $i\in I:=\{0,\dots,N\}$, see e.g.\ \cite[Section III.1]{Dav96}. For each $i\in I$, call $E_i\in\mathfrak{A}_i$ the respective central projection, so that $\mathfrak{A}_i = E_i\mathfrak{A}E_i$. Then, note that $\one_X=\sum_{i=0}^N E_i$, and we can suppose that $\mathfrak{A}_0=\C E_0$ with $E_0=\beta\iota\iota^*$ for some positive real number $\beta$ thanks to the haploid condition. As a consequence, by the convention right above Definition \ref{def:rigid_haploid_alg}, the rigidity of $A$ is realized with $\varepsilon=\beta\iota^*$. Moreover, we can write any $T\in\mathfrak{A}$ as a sum $T=\sum_{i=0}^{N}T_i$ with $T_i\in\mathfrak{A}_i$. Hence
	\begin{equation}
		\Phi(T)=\sum_{i,j=0}^N\Phi_{ij}(T_j)
	\end{equation}
where $\Phi_{ij}: \mathfrak{A}_j\to\mathfrak{A}_i$ is defined by 	$\Phi_{ij}(\cdot):= E_i\Phi\restriction_{\mathfrak{A}_j}(\cdot)E_i$.

As an intermediate step, we prove that $\Phi^2$ is a strictly positive linear map. Note that $\Phi$ is a positive linear map by definition and thus each $\Phi_{ij}$ is too. For each $j\in I$, let $T_j\in\mathfrak{A}_j$ be any non-zero positive element, then we have that 
	$$
	\Phi_{0j}(T_j)=E_0  (T_j \ast \one_X) E_0
	= \iota\varepsilon m(T_j\otimes \one_X) m^* \iota\varepsilon
	=0
	$$
	if and only if (note that $\varepsilon m(T_j\otimes \one_X) m^* \iota \in \Hom(\id_\mathcal{C},\id_\mathcal{C}) \cong \C$ and $\iota \varepsilon = \beta\iota\iota^* = E_0$)
	$$
	\varepsilon m(T_j\otimes \one_X) m^* \iota
	=\varepsilon m(\sqrt{T_j}\otimes \one_X)[\iota^* m(\sqrt{T_j}\otimes \one_X)]^*
	=0	\,.
	$$
The latter equation (recall that $\varepsilon = \beta \iota^*$ with $\beta>0$) is equivalent to $\varepsilon m(\sqrt{T_j}\otimes \one_X)=0$, which is not possible in our setting thanks to the rigidity of $A$. 
Indeed, suppose by contradiction that this is not the case. Rigidity means that $e_A=\varepsilon m$ is part of a solution $e_A, i_A$ of the conjugate equations \eqref{eq:conj_eqns} for $X$ and $\overline{X}=X$. Thus, $e_A(\sqrt{T_j}\otimes \one_X)=0$ implies that
	\begin{equation}
		\begin{split}
			0
			&=
			l_X(e_A\otimes \one_X)a_{X,X,X}[\sqrt{T_j}\otimes (\one_X\otimes \one_X) ](\one_X\otimes i_A)r_X^{-1}\\
			&=
			l_X(e_A\otimes \one_X)a_{X,X,X}(\sqrt{T_j}\otimes i_A)r_X^{-1} \\
			&=
			l_X(e_A\otimes \one_X)a_{X,X,X}(\one_X\otimes i_A) r_X^{-1}r_X(\sqrt{T_j}\otimes \one_{\id_\mathcal{C}}) r_X^{-1} \\
			&=
			r_X(\sqrt{T_j}\otimes \one_{\id_\mathcal{C}}) r_X^{-1}  
		\end{split}
	\end{equation}
which is only possible if $T_j = 0$.
Therefore, it must be $\Phi_{0j}(T_j) \geq 0$ and $\Phi_{0j}(T_j)\not=0$ for each non-zero positive element $T_j\in\mathfrak{A}_j$. Moreover, $\Phi_{0j}(T_j)=\alpha_j E_0$ for some positive real number $\alpha_j$ because $\mathfrak{A}_0$ is one-dimensional thanks to the haploid condition.
Therefore, we conclude that
	\begin{equation}
		\begin{split}
			(\Phi^2)_{ij}(T_j) 
			 &=
			\sum_{k=0}^N \Phi_{ik}(\Phi_{kj}(T_j)) \\
			 &\geq 
			\Phi_{i0}(\Phi_{0j}(T_j)) \\
			 &=
			\alpha_j\Phi_{i0}(E_0) \\
			 &=
			\alpha_j E_i (E_0 \ast \one_X) E_i \\
			 &=
			\alpha_j \beta E_i (\iota\iota^* \ast \one_X) E_i \\
 			 &=
			 \alpha_j\beta E_i \one_X E_i \\
			 &=
			\alpha_j\beta E_i 
			\qquad\forall i,j\in I
		\end{split}
	\end{equation}
where we have used the positivity of $\Phi_{ij}:\mathfrak{A}_j \to \mathfrak{A}_i$ for all $i,j\in I$ and the properties of the convolution. Now, if $T=\sum_{i=0}^{N}T_i$ is positive and non-zero in $\mathfrak{A}$, then each $T_i$ is positive, and at least one of them, say $T_j$, is non-zero. Then, 
	\begin{equation}
		\begin{split}
				\Phi^2(T) &= \sum_{i,k=0}^N (\Phi^2)_{ik}(T_k) \\
				&\geq \sum_{i=0}^N (\Phi^2)_{ij}(T_j) \\
				&\geq \sum_{i=0}^N \alpha_j \beta E_i \\
				&= \alpha_j \beta \one_X 
		\end{split}
	\end{equation}
which implies that $\Phi^2:\mathfrak{A} \to \mathfrak{A}$ is a strictly positive map. 
\smallskip

We now prove that the positive map $\Phi$ is irreducible. As the dimension of the Hilbert space on which $\mathfrak{A}$ is realized is at least $n=2$ (unless $X \cong \id_\mathcal{C}$ in which case the statement of the theorem is trivial), then $(\id_\mathfrak{A} + \Phi)^{n-1}(T) > 0$ for $T$ positive and non-zero as above.
Indeed, if $n\geq3$ then $(\id_\mathfrak{A} + \Phi)^{n-1}(T) \geq \Phi^2(T)>0$.
If $n=2$, then $\mathfrak{A} = \mathfrak{A}_0 \oplus \mathfrak{A}_1$ with $\mathfrak{A}_1 \cong \C$ as well, namely $\mathfrak{A}_1 = \C E_1$.
In this case, it is enough to check that $(\id_\mathfrak{A} + \Phi)(T) > 0$ on the two projections $E_0$ and $E_1$. 
If $T=E_0$, observe that $\Phi(E_0) = \beta (\iota \iota^* \ast \one_X) = \beta \one_X$, where $\beta >0$, hence $E_0 + \Phi(E_0) > 0$.
If $T=E_1$, then $(\id_\mathfrak{A} +\Phi)(E_1) = \delta_0 E_0 +(1+ \delta_1) E_1$ for some real numbers $\delta_0,\delta_1 \geq 0$, which is strictly positive unless $\delta_0 = 0$. 
But, if $\delta_0 = 0$, then $\Phi(E_1) = \delta_1 E_1$, which is a contradiction as it would imply that $\Phi^2(E_1) = (\delta_1)^2 E_1$, which is not strictly positive.
Summing up, we have shown that $\Phi$ is an irreducible positive map. 
	
By Proposition \ref{prop:perron_frobenius}, there exists a unique (in the strong sense specified in the proposition) strictly positive eigenvector $T\in\mathfrak{A}$ with positive eigenvalue $\lambda$ for $\Phi$, namely such that $\Phi(T) = T\ast \one_X = \lambda T$. But also $T \ast T\in\mathfrak{A}$ is an eigenvector with the same eigenvalue, indeed $\Phi(T\ast T) = (T\ast T) \ast \one_X = T\ast (T \ast \one_X) = \lambda (T\ast T)$. Moreover, $T\ast T$ is positive by the definition of the convolution and non-zero. Indeed, $T\ast T = 0$ would imply $m (\sqrt{T} \otimes \sqrt{T}) = 0$, which is not possible as it would contradict the unit property of $A$.
By the uniqueness part of Proposition \ref{prop:perron_frobenius}, we conclude that $T\ast T = \gamma T$, for some positive real number $\gamma$.
In other words, the pair $(T,\gamma)$ is a solution of \eqref{eq:specialness_strictly_positive} and thus $(S:=\sqrt{T},\gamma)$ is a solution of \eqref{eq:specialness_mu_S}. In conclusion, the deformed triple $A_S:=(X,  m_S, \iota_S)$ with $ m_S:=S^{-1} m(S\otimes S)$ and $\iota_S:=S^{-1}\iota$, where $S$ happens to be strictly positive in $\mathfrak{A}$ and not only invertible, satisfies specialness and thus it is the desired Q-system equivalent to $A = (X,  m, \iota)$.
	
Conversely, suppose that the haploid algebra $A=(X, m,\iota)$ is equivalent to a Q-system. This means that there exists a deformed triple by the same formulae as above $A_S=(Y,  m_S, \iota_S)$, for an invertible $S\in\Hom(X,Y)$, which satisfies the $C^*$-Frobenius property, or equivalently specialness by Lemma \ref{lem:frobenius_iff_special}, thanks to the haploid assumption.
We have already observed in Remark \ref{rem:CstarFrob_is_rigid_and_real} that a $C^*$-Frobenius algebra is rigid, and in Remark \ref{rem:equiv_and_unitary_equiv} that rigidity is preserved under equivalence. Thus, $A$ is necessarily rigid.
	
For the second claim, note that any algebra is equivalent to a normalized one by Remark \ref{rem:associative_C-algebra_equiv_normalized}. Suppose that $A= (X,  m,  \iota)$ is haploid and rigid. Then, the Q-system $A_S= (X,  m_S, \iota_S)$ just constructed, can be considered to be normalized, by substituting the positive isomorphism $S\in\Hom(X,X)$ with an appropriate positive multiple, still denoted by $S$. By the haploid condition and by the strict positivity of $S$, $\iota_S=\alpha\iota$ for a positive real number $\alpha$. If $A$ is also normalized, then 
	$$
	\one_{\id_\mathcal{C}}=\iota_S^*\iota_S=\alpha^2\iota^*\iota=\alpha^2 \one_{\id_\mathcal{C}}
	$$
and thus $\alpha=1$, i.e., $\iota_S=\iota$. Conversely, if $\iota_S=\iota$, then $A$ is obviously normalized.
	
Finally, if $A$ is just haploid and rigid, then the standardness of $A_S$ follows from the fact that any haploid Q-system in a $C^*$-tensor category is standard,
see Remark \ref{rem:std_Qsys}. Thus, the proof of the theorem is complete.	
\end{proof}

We continue with some remarks and consequences:

\begin{rem}  \label{rem:rigid_iff_Frobenius} 
It follows by \cite[Lemma 3.7]{FRS02}, cf.\ \cite[Proposition 8]{FS08} (not necessarily in the $C^*$ context) that an algebra $A = (X,m,\iota)$ admits a \emph{Frobenius algebra structure} if and only if it is \emph{rigid}. The Frobenius algebra structure (namely the comultiplication solving the Frobenius relations) is uniquely determined by the counit, there denoted by $\varepsilon$. In the $C^*$ context, Theorem \ref{theo:haploid_rigid_iff_Q-system} can be rephrased as follows. If $A$ is haploid and rigid, and we set $\varepsilon := \iota^*$, then $A$ can be deformed to an equivalent algebra $A_S$ (still haploid and rigid) where the unique solution (the Frobenius algebra structure) is given by the comultiplication $\Delta := m^*$. In particular, the solution is a \emph{$C^*$-Frobenius algebra structure}.
\end{rem}

\begin{rem} 
Under the mild assumption that the object $X$ is rigid as an object in $\mathcal{C}$, it follows by \cite[Lemma 3.6]{Yam04b} that the rigidity of $A$ as an algebra is actually equivalent (not only implies, as we observed and used in the proof of Theorem \ref{theo:haploid_rigid_iff_Q-system}) to the faithfulness of the state $\iota^* \Phi(\cdot) \iota : \mathfrak{A} \to \Hom(\id_\mathcal{C},\id_\mathcal{C}) \cong \C$ on the finite dimensional $C^*$-algebra $\mathfrak{A} = \Hom(X,X)$, where $\Phi:\mathfrak{A}\to\mathfrak{A}$ is as in \eqref{eq:map_Phi_def}.
\end{rem}

Theorem \ref{theo:haploid_rigid_iff_Q-system} provides an answer to the following natural questions on the \emph{unitarizability} of categorical structures (algebras, modules) in the $C^*$-tensor category context:

\begin{rem}
	The idea of deforming a (not necessarily haploid) algebra $A$ in $\mathcal{C}$, to a special algebra is present at the end of \cite[Section 3.1]{BKLR15}, where an iterative method is proposed. Assuming the haploid property, the proof of Theorem \ref{theo:haploid_rigid_iff_Q-system} provides a solution of $T \ast T = \gamma T$, see \eqref{eq:specialness_strictly_positive}, for some positive real number $\gamma$, where $T$ is strictly positive in $\mathfrak{A} = \Hom(X,X)$. Namely, $T$ fulfills the desired \lq\lq regularity condition\rq\rq mentioned in \cite[Page 20]{BKLR15}.
\end{rem}

\begin{rem}
	In \cite[Question 3.7]{CHJP22}, it is asked whether every arbitrary Frobenius algebra in a \emph{unitary} tensor category (a semisimple rigid $C^*$-tensor category with simple unit in the terminology of \cite[Definition 2.11]{CHJP22}) is equivalent to a Q-system. By Remark \ref{rem:rigid_iff_Frobenius}, rigidity for an algebra (not necessarily haploid) is equivalent to the existence of a Frobenius algebra structure. Hence, assuming the haploid property, Theorem \ref{theo:haploid_rigid_iff_Q-system} provides a positive answer to \cite[Question 3.7]{CHJP22}. Below, we show that the answer becomes negative if the haploid assumption is dropped, see Example \ref{ex:CE1}.
\end{rem}

\begin{rem}
	By \cite[Theorem 3.1 and Proposition 3.1]{Ost03}, every semisimple indecomposable (right) module category $\mathcal{M}$ over a \emph{fusion} category $\mathcal{C}$ 
	(a semisimple rigid tensor category with simple tensor unit and \emph{finitely} many equivalence classes of simple objects) is equivalent to $\Mod_\mathcal{C}(A)$ 
	for some 
	rigid haploid algebra $A$ in $\mathcal{C}$. 
	Hence, our Theorem \ref{theo:haploid_rigid_iff_Q-system} gives a positive answer to \cite[Question 12]{Reu19}, namely whether every semisimple module category over a \emph{unitary} fusion category is equivalent to a $C^*$-module category, 
in the case of $\mathcal{M}$ indecomposable. 
\end{rem}

We summarize this latter discussion in:

\begin{cor}  \label{cor:unitarizability_module_cats_over_fusion}
	Let $\mathcal{M}$ be a semisimple indecomposable module category over a unitary fusion category $\mathcal{C}$. 
	Then, $\mathcal{M}$ is equivalent to $\Mod_\mathcal{C}(A)$ for some haploid Q-system $A$ in $\mathcal{C}$. Consequently, every (direct sum of) semisimple indecomposable module category $\mathcal{M}$ over a unitary fusion category $\mathcal{C}$ is unitarizable (uniquely up to unitary module equivalence).
\end{cor}

\begin{proof}
	An equivalence of algebras in $\mathcal{C}$ entails an equivalence of the respective module categories. Thus, the first claim follows from \cite[Theorem 3.1 and Proposition 3.1]{Ost03} and from Theorem \ref{theo:haploid_rigid_iff_Q-system}. 
	For the second claim, it is enough to observe that modules over special algebras (in particular over haploid Q-systems) are equivalent to unitary modules, see \cite[Lemma 3.22]{BKLR15}, and the latter form a unitary (in fact, $C^*$) category, see \cite[Section 2.2 and 2.3]{GY20} and \cite[Remark 6.2]{NY16}.
\end{proof}

By the techniques used in the proof of Theorem \ref{theo:haploid_rigid_iff_Q-system}, we can prove the following:

\begin{theo} \label{theo:algebras_equiv_as_Qsystems}
	Let $\mathcal{C}$ be a $C^*$-tensor category (with simple tensor unit). Then, any equivalence between normalized haploid (hence standard) Q-systems in $\mathcal{C}$ is necessarily unitary. Furthermore, there is a one-to-one correspondence between:
		\begin{itemize}
			\item[(i)] equivalence classes of rigid haploid algebras in $\mathcal{C}$;
			\item[(ii)] unitary equivalence classes of normalized haploid (hence standard) Q-systems in $\mathcal{C}$.
		\end{itemize}
\end{theo}

\begin{proof} 
	Let $A=(X,m,\iota)$ and $B=(Y,n,j)$ be two normalized haploid Q-systems. 
	Let $S\in\Hom(X,Y)$ be an isomorphism (not necessarily unitary) realizing the equivalence between $A$ and $B$ as algebras, see Definition \ref{def:equiv_and_unitary_equiv}. Recall the convolution $\ast$ introduced for the proof of Theorem \ref{theo:haploid_rigid_iff_Q-system}. Set $T:=S^{-1}(S^{-1})^*$ and note that $T \ast T = m(T\otimes T)m^*=\lambda T$ for some positive real number $\lambda$ by the specialness of the Q-system $B$. 
	
	Now, define a linear map $\Phi':\Hom(X,X)\to \Hom(X,X)$ by $\Phi'(R):=\one_X\ast R$ for all $R\in\Hom(X,X)$. Proceeding as in the proof of Theorem \ref{theo:haploid_rigid_iff_Q-system}, it is easy to see that the rigidity of $A$ implies that $(\Phi')^2$ is a strictly positive map. 
	Define also the linear map $\Psi:\Hom(X,X)\to \Hom(X,X)$ by $\Psi(R):=T\ast R$ for all $R\in\Hom(X,X)$. As $T=S^{-1}(S^{-1})^*$ is strictly positive, we have that $T\geq \delta \one_X$ for some positive real number $\delta$ and thus $\Psi(R)\geq \delta\Phi'(R)$ for all positive $R\in\Hom(X,X)$. It follows that $\Psi^2(R)\geq \delta^2(\Phi')^2(R)$ for all positive $R\in\Hom(X,X)$, and thus $\Psi^2$ is strictly positive because $(\Phi')^2$ is. 
	Again, proceeding as in the proof of Theorem \ref{theo:haploid_rigid_iff_Q-system}, we can use that $\Psi^2$ is strictly positive to prove that $\Psi$ is an irreducible positive map.
	By definition, $\Psi$ has $T$ as a strictly positive eigenvector with positive eigenvalue $\lambda$. By Proposition \ref{prop:perron_frobenius}, any other non-zero and positive eigenvector of $\Psi$ must be a positive multiple of $T$.
	Let $\Phi:\Hom(X,X)\to \Hom(X,X)$ be the irreducible positive linear map defined in the proof of Theorem \ref{theo:haploid_rigid_iff_Q-system} as $\Phi(R):=R\ast \one_X$ for all $R\in\Hom(X,X)$. Then, 
	$$
	\Psi(\Phi(T))=T\ast \Phi(T)= T\ast (T\ast \one_X)=(T\ast T)\ast \one_X
	=\lambda T\ast \one_X =\lambda \Phi(T) \,.
	$$
	Hence, $\Phi(T)$ is a non-zero positive eigenvector of $\Psi$ with eigenvalue $\lambda$ and thus there exists a positive real number $\alpha$ such that $\Phi(T)=\alpha T$. Again by Proposition \ref{prop:perron_frobenius} applied to $\Phi$, observing that $\one_X$ is a strictly positive eigenvector of $\Phi$ by the specialness of the Q-system $A$, we conclude that $S^{-1}(S^{-1})^*=T=\beta \one_X$ for some positive real number $\beta$. Using the polar decomposition of $(S^{-1})^*$, we get that $(S^{-1})^*=U\sqrt{T}=U\sqrt{\beta} \one_X=\sqrt{\beta} U$ for some unitary $U\in \Hom(X, Y)$. Then, $S=(\sqrt{\beta})^{-1}U$. By the normalization condition, $\one_{\id_\mathcal{C}} = j^* j = \iota^* S^* S \iota = \beta^{-1} \iota^* U^* U \iota = \beta^{-1} \one_{\id_\mathcal{C}}$. Then, $\beta = 1$ and $S=U$ is a unitary equivalence between the normalized haploid Q-systems $A$ and $B$. Finally, the one-to-one correspondence follows from Theorem \ref{theo:haploid_rigid_iff_Q-system} and from the first claim.
\end{proof}

\begin{rem} 
In \cite[after Theorem 2.9]{NY18}, the authors affirm that any isomorphism, see \cite[Definition 2.4]{NY18}, between irreducible (i.e., haploid, see \cite[Remark 2.7(i)]{NY18}) Q-systems is unitary up to a scalar factor. Nevertheless, their definition of isomorphism between $C^*$-Frobenius algebras is stronger than the definition of equivalence for algebras considered here, see Definition \ref{def:equiv_and_unitary_equiv}.
\end{rem}

Now, we provide a counterexample to the statement of Theorem \ref{theo:haploid_rigid_iff_Q-system}, if the haploid condition is removed (Example \ref{ex:CE1}). We also show that the haploid property alone does not imply rigidity (Example \ref{ex:CE2}).

\begin{ex}  \label{ex:CE1}
	Let $\mathcal{C} := \Hilb$ be the (trivial) rigid semisimple $C^*$-tensor category of finite dimensional complex Hilbert spaces, with tensor unit $\id_\mathcal{C} = \C$. We shall find a non-haploid Frobenius algebra in $\C$, which is \emph{not} equivalent to a Q-system.
	
	Let $X$ be the two-dimensional Hilbert space generated by the $2\times 2$ matrices $I := \begin{pmatrix} 1 & 0 \\ 0 & 1 \\ \end{pmatrix}$ and $N := \begin{pmatrix} 0 & 1 \\ 0 & 0 \\ \end{pmatrix}$. The multiplication morphism $m:X\otimes X \to X$ is defined by the ordinary matrix multiplication $(aI + bN)(cI + dN) = acI + (ad+bc)N$ for $a,b,c,d\in\C$. The unit morphism $\iota:\C \to X$ is defined by $a \mapsto aI$ for $a\in\C$. This algebra in $\mathcal{C}$ is \emph{not} haploid, as $\Hom(\C,X) \cong \C^2$, where we recall that the morphisms in $\Hilb$ are just the linear maps. Then $\varepsilon (aI+bN) := a+b$ defines a linear functional $\varepsilon:X\to\C$ such that the composition $\varepsilon m : X\otimes X \to \C$ is a non-degenerate pairing, i.e., the functional $\varepsilon$ is a Frobenius form.
	Recall that rigidity is equivalent to the existence of a Frobenius structure by \cite[Lemma 3.7]{FRS02}, see Remark \ref{rem:rigid_iff_Frobenius}.
	However, $(X,m,\iota)$ is \emph{not} equivalent to a $C^*$-Frobenius algebra (hence nor to a special one), because $X$ is not a $C^*$-algebra (it is generated by $I$ and by a nilpotent $N$) with respect to any involution, see \cite[Lemma 2.2]{NY18}. 
\end{ex}

\begin{ex}  \label{ex:CE2}
	Let $\mathcal{C} := \Rep(\Z_2)$ be the (symmetric) rigid semisimple $C^*$-tensor category of finite dimensional unitary representations of $\Z_2 = \Z/{2\Z}$, with tensor unit $\id_\mathcal{C} = \C$. We shall find a haploid algebra in $\mathcal{C}$ which is not rigid, i.e., not Frobenius.
	
	Let $X$ be the two-dimensional Hilbert space generated by $I$ and $N$ as in the previous example, with the inner product $(aI+bN | cI+dN) := \Tr ((aI+bN)^*(cI+dN)) = 2 \overline{a}c + \overline{b}d$, where $\Tr$ is the ordinary non-normalized trace. Endow $X$ with the unitary $\Z_2$-action given by $-1 \cdot (aI+bN) := S (aI+bN) S^{-1} = aI-bN$, where $S := \begin{pmatrix} 1 & 0 \\ 0 & \!\!\!\! -1 \\ \end{pmatrix}$. Hence $X$ is an object in $\mathcal{C}$. Now, $(X,m,\iota)$ is an algebra in $\mathcal{C}$ with the same operations defined before. But it is haploid in this case, as the linear map $b \mapsto b N$ does not intertwine the $\Z_2$-actions on $\C$ and $X$. The adjoint of the unit morphism $\iota^* : X \to \C$ is given by $\iota^*(aI+bN) = 2a$. The haploid algebra $(X,m,\iota)$ does not admit a Frobenius structure with counit $\iota^*$, as the composition $\iota^* m : X \otimes X \to \C$ is degenerate. Indeed, $N^2 = 0$ and $N I = N$ is sent to $\iota^*(N) = 0$, while $N$ is not zero.
\end{ex}

\section{General applications to chiral CFT}	\label{sec:applications}

In this section, we apply the tools just developed to obtain general results in the theory of vertex operator (super)algebras and (graded-local) conformal nets. In Section \ref{sec:unitarity_schellekens_list} and Section \ref{sec:superconformal_VOAs}, we draw consequences for various notable chiral conformal field theory (chiral CFT) models.

In the following exposition, we assume that the reader is familiar with the two axiomatic frameworks and with the relations between them. Nevertheless, for the less expert readers, we give the necessary references here below and across the various sections whenever appropriate, and then we recall standard notations.

\begin{genref}
The theory of \textbf{vertex operator algebras} (VOAs) can be found e.g.\ in \cite{Kac01,FLM88,LL04,FHL93}. The last two are also a classical introduction to the theory of VOA modules and, just the last one, to the theory of intertwining operators. VOA modules and intertwining operators are the building blocks of the braided tensor category theory of VOA representations, which was developed in the series of papers \cite{HL95I, HL95II, HL95III, Hua95, Hua08}. See \cite[Section 2]{HKL15} for a brief review. 
\cite{Kac01} and \cite{Li96} include the theory of \textbf{vertex operator superalgebras} (VOSAs), which will be necessary in some of our applications. 

Regarding the \textbf{unitary theory} for VOAs, we refer the reader to \cite{DL14} and to \cite[Chapter 5]{CKLW18}; whereas we refer to \cite{AL17a}   for the unitary theory of VOSAs, see also \cite{CGH} and \cite[Chapter 2]{Gau21}. Turning back to the representation theory for VOAs, the unitary modular tensor structure for VOA representations was established in \cite{Gui19I, Gui19II}. 

For the theory of (local) \textbf{conformal nets} see e.g.\ \cite[pp.\ 1–2 and Section 3.3]{CKLW18}. Instead, \textbf{graded-local conformal nets}, see \cite{CKL08}, are the operator algebraic counterpart of VOSAs.

The correspondence between VOAs and conformal nets is established in \cite{CKLW18}. Moreover, \cite{Gui21a} and \cite{Gui20} study the relation between the respective representation theories for VOAs and conformal nets. 
\end{genref}

The (locally normal) representations of an irreducible conformal net $\A$ form a braided $C^*$-tensor category with simple tensor unit, often denoted by $\Rep(\A)$, see \cite{FRS89, FRS92}, based on DHR superselection theory \cite{DHR71, DHR74,Haa96}. An equivalent approach through Connes' fusion is discussed e.g.\ in \cite{Was95,Was98}, see also \cite[Section 6]{Gui21a} and \cite{Gui21b}. 
\textit{Finite index} representations, see e.g.\ \cite[Section 2.1]{GL96} and \cite[Section 2.2]{Car04}, determine a full subcategory of $\Rep(\A)$, sometimes denoted by $\Repf(\A)$. 
We will often deal with irreducible conformal nets that are \textbf{completely rational}, see \cite[Definition 8]{KLM01}.  By \cite[Theorem 4.9]{LX04} and \cite[Theorem 5.4]{MTW18}, a conformal net is completely rational if and only if it has finitely many inequivalent irreducible representations and all of them have finite index or, equivalently, if and only if it has finite $\mu$-index, see \cite[Theorem 5.3]{LX04}. 
If a conformal net $\mathcal{A}$ is completely rational then every representation in $\Rep(\A)$ is a (possibly infinite) direct sum of irreducible ones and the objects in $\Repf(\A)$ are exactly the finite direct sums of irreducibles in $\Rep(\A)$. Moreover,  $\Repf(\A)$ is a unitary modular tensor category, see \cite[Corollary 37 and Corollary 39]{KLM01}.

Recall that every graded-local conformal net $\A$ has a grading unitary, i.e., a unitary self-adjoint automorphism of the defining vacuum Hilbert space of $\A$. The eigenspaces of the operator induced by the adjoint action, of eigenvalues $1$ and $-1$ are called $\A_\parzero$, the \textbf{even} part, and $\A_\parone$, the \textbf{odd} part, of $\A$ respectively.
In the same way, a VOSA $V$ has a parity automorphism, i.e., a VOSA automorphism of order two, such that $V=V_\parzero\oplus V_\parone$, where $V_\parzero$ and $V_\parone$ are the eigenspaces of eigenvalues $1$ and $-1$ respectively. They are called the \textbf{even} and the \textbf{odd} part of $V$ respectively. Then, a VOA is simply a VOSA with trivial parity automorphism, i.e., $V_\parone=\{0\}$. 

A VOSA $V$ is said to be \textbf{simple} if the only ideals are $\{0\}$ and $V$ itself, which corresponds to $V$ being irreducible as $V$-module. Instead, a VOSA $V$ is said to be \textbf{of CFT type} if $V_0=\C\Omega$ and $V_n=\{0\}$ for all negative $n$. Note that the latter condition is actually automatic for VOAs with just $V_0=\C\Omega$, see \cite[Remark 4.5]{CKLW18}. Furthermore, simplicity and the CFT type condition are equivalent for unitary VOSAs, see \cite[Proposition 5.3]{CKLW18} and \cite[Proposition 2.14]{Ten19b}, cf.\ also \cite{CGH} and \cite[Proposition 2.3.6]{Gau21}. 
Recall that a VOA $V$ is said to be \textbf{self-dual} (or self-contragredient) if, as $V$-module, it is isomorphic to its \textit{contragredient module}. This is equivalent to the existence of a non-degenerate \textit{invariant} bilinear form on $V$, see e.g.\ \cite[Section 4.3]{CKLW18}. Consequently, any self-dual VOA of CFT type is simple by \cite[Proposition 4.6(iv)]{CKLW18}. Moreover, any unitary VOA is self-dual by definition, see e.g.\ \cite[Section 5.1]{CKLW18}.

\begin{defin}
	A VOA $V$ is said to be \textbf{regular} if every weak $V$-module is a direct sum of irreducible (ordinary) $V$-modules. $V$ is said to be \textbf{strongly rational} (or sometimes strongly regular) if, in addition, it is of CFT type and self-dual.
\end{defin}	

It is well known that regularity implies \textit{rationality} and $C_2$-\textit{cofiniteness}. On the other hand, the two latter properties together imply regularity under the assumption of CFT type. In particular, rationality implies that there are only finitely many irreducible admissible $V$-modules and every admissible $V$-module is actually ordinary. See \cite{ABD04} and references therein for definitions and results. Note that a strongly rational VOA is also simple, rational and $C_2$-cofinite. Lastly, if $V$ is strongly rational, then the $V$-modules form a modular tensor category $\Rep(V)$, see \cite{Hua08} and references therein.

\begin{defin}   \label{defin:complete_unitarity}
A regular VOA $V$ of CFT type is said to be \textbf{completely unitary} if it is \emph{unitary}, if every $V$-module is unitarizable (in this case, $V$ is said to be \textit{strongly unitary}) and if the \textit{invariant sesquilinear form} $\Lambda$ on the categorical tensor product of $V$-modules, as defined in \cite[Section 6.2]{Gui19II}, see also \cite[Definition 1.7]{Gui21c}, is positive definite. 
\end{defin}	

\begin{rem}  \label{rem:complete_unitarity}
Let $V$ be as in Definition \ref{defin:complete_unitarity}. Note that $V$ is always strongly rational. By strong unitarity, the $C^*$-category $\Repu(V)$ of unitary $V$-modules, see \cite[Section 2.4]{Gui19I}, is linearly equivalent to $\Rep(V)$. Moreover, the invariant sesquilinear form $\Lambda$ turns $\Repu(V)$ into a unitary modular tensor category by \cite[Theorem 7.9]{Gui19II}. In particular, the forgetful functor from $\Repu(V)$ to $\Rep(V)$ is an equivalence of modular tensor categories.
\end{rem} 

\begin{rem}   \label{rem:uniqueness_unitary_structure}
The unitary structure of a simple unitary VOA is unique up to unitary isomorphism. Furthermore, the unitary structure is really unique, that is not just up to  isomorphism, if every automorphism of the VOA is unitary, see the discussion just after \cite[Proposition 5.19]{CKLW18}. 
\end{rem}

The following is a result of independent interest, which will be needed in some of the applications in Section \ref{subsec:applications_VOAs}.

\begin{prop}   \label{prop:uniqueness_module_with_zero_weight}
	Let $V$ be a simple unitary VOA and $M$ be an irreducible unitary (ordinary) $V$-module. Then the conformal Hamiltonian $L^M_0$ of the module $M$ has only non-negative eigenvalues. Moreover, if there exists a non-zero vector $v\in M$ of conformal weight $d_v=0$, then $M$ has a VOA structure, which makes it a VOA isomorphic to $V$. Hence, in the latter case, the module $M$ is equivalent to the adjoint $V$-module $V$. 
\end{prop}

\begin{proof}
	For any $a\in V$, let $Y^M(a,z)=\sum_{n\in\Z}a^M_{(n)}z^{-n-1}$ denote any vertex operator on the $V$-module $M$. Let $\Omega$ and $\nu$ be the vacuum and the conformal vector of $V$. Moreover, we use the symbol $\scalar_M$ for the scalar product making $M$ into a unitary $V$-module.  
	
	The endomorphisms $L^M_n$, $n \in \mathbb{Z}$ defined by $Y^M(\nu,z)=\sum_{n\in\Z}L^M_{n}z^{-n-2}$ give a representation of the Virasoro algebra 
	on $M$ with central charge $c$ (the cental charge of $V$). This representation is unitary because $M$ is a unitary $V$-module and consequently the eigenvalues 
	of $L^M_0$ are non-negative real numbers, see \cite{KR87}, cf.\ also \cite[Proposition 1.7]{Gui19I}.	
	
	Let us now assume that $L^M_0v =0$ for some non-zero $v \in M$.  Then $L^M_{1}v=0$ because $L^M_0L^M_{1}v = -L^M_{1}v$. Thus, we have also that $L^M_{-1}v=0$ as
	$$
	(L^M_{-1}v|L^M_{-1}v)_M=(v|L^M_1L^M_{-1}v)_M=(v|2L^M_0v)_M=0 \,.
	$$
	Now, let $\phi: V \to M$ be the linear map defined by $\phi(a)=a^M_{(-1)}v$. By  \cite[Proposition 4.2.1]{LL04} and \cite[Theorem 1.5]{DSK06} $\phi$ is an injective
	map. 
	By the existence theorem \cite[Theorem 4.5]{Kac01} $M$ have a VOA structure with vacuum vector $\Omega^M :=v$, conformal vector 
	$\nu^M : = \nu^M_{(-1)}v$ and state-field correspondence $m \mapsto \tilde{Y}(m,z)$, $m \in M$ satisfying 
	$\tilde{Y}(\phi(a),z) = Y^M(a,z)$, for all $a\in V$.  It follows that 
	$$\phi(a_{(n)}b) = (a_{(n)}b)^M_{(-1)}v =  a^M_{(n)} \phi(b)$$ 
	for all $a,b \in V$ and all $n \in \Z$.  In particular, $\phi(V)$ is a non-zero V-submodule of $M$ and hence, by the irreducibility of $M$, $\phi(V)=M$, i.e.\  $\phi$ is also 
	surjective. Moreover, $\phi$ is a VOA isomorphism.
	\end{proof}

\subsection{Applications to VOAs and conformal nets} \label{subsec:applications_VOAs}

The next important result, see Theorem \ref{theo:unitarity_voa_extensions}, is a first application of Theorem \ref{theo:haploid_rigid_iff_Q-system} and of Theorem \ref{theo:algebras_equiv_as_Qsystems} to the VOA theory. We first recall the definition of a (simple CFT type) VOA extension, see e.g.\ \cite{HKL15}.  For the related definition of \emph{vertex subalgebra}, see e.g.\ \cite[Section 4.3]{Kac01}.  For \emph{unitary (vertex) subalgebras}, see \cite[Section 5.4]{CKLW18} and also \cite{CGH19}. 

\begin{defin}
\label{def:VOA extension}
A simple VOA $U$ of CFT type is called a \textbf{simple CFT type VOA extension} of a simple VOA $V$ of CFT type if the latter, namely $V$, is a vertex subalgebra of $U$ with the same conformal vector. Moreover, $U$ is said in addition to be a simple \textbf{unitary} CFT type VOA extension of $V$ if both VOAs are unitary and if the unitary structure of $U$ restricts to the one of $V$, i.e., $V$ is a unitary subalgebra of $U$.
\end{defin}

\begin{theo}   \label{theo:unitarity_voa_extensions}
	Let $V$ be a regular completely unitary VOA of CFT type and $U$ be a simple CFT type VOA extension of $V$. Then, $U$ is automatically unitary as simple CFT-type VOA extension of $V$, and it is also a completely unitary VOA. 
	Furthermore, if $\phi:U_1\to U_2$ is a VOA isomorphism between simple unitary CFT-type VOA extensions of $V$, which restricts to the identity map on $V$, then $\phi$ is unitary and $U_1$ and $U_2$ are unitarily isomorphic. In particular, the unitary structure on $U$ extending the given one on $V$ is unique (not just up to unitary isomorphism, see Remark \ref{rem:uniqueness_unitary_structure}).
\end{theo}

\begin{proof}
Throughout the proof, we use $Y_U$ and $Y^M$ for the state-field correspondences of the VOA $U$ and of a $U$-module $M$ respectively.	
	
Note that, since $V$ is strongly unitary, $U$ is a \textit{preunitary} VOA extension of $V$ in the sense of \cite[Section 2.1]{Gui21c}, i.e., it has a scalar product whose restriction to $V$ is the one realizing the unitary structure of $V$. Furthermore, we can obviously associate to $U$ a $V$-module $(W,Y^W)$ such that $W=U$ as $\C$-vector spaces and $Y^W(a,z)=Y_U(a,z)$ for all $a\in V$. Therefore, $U$ determines a normalized haploid commutative algebra $A=(W,\mu,\iota)$ in $\Repu(V)$ with trivial twist, where $\mu\in\Hom(W\otimes W,W)$ is the map associated to $Y^W$ and $\iota\in\Hom(V,W)$ is the obvious inclusion map of $V$ into $W$, see \cite[Theorem 3.2]{HKL15} for details and Remark \ref{rem:complete_unitarity}, cf.\ also \cite[Proposition 2.1]{Gui21c}. Note also that $U$ is irreducible as $U$-module by simplicity. Therefore, $A$ is rigid by \cite[Lemma 1.20, cf.\ Theorem 5.2]{KO02}, cf.\ also \cite[Theorem 3.6 and Remark 3.7]{HKL15}.

Now, $A$ satisfies all the hypotheses of Theorem \ref{theo:haploid_rigid_iff_Q-system} and thus there exists a strictly positive isomorphism $S\in\Hom(W,W)$ such that $A_S=(W, \mu_S, \iota)$ is a normalized haploid 
Q-system in $\Repu(V)$ equivalent to $A$, where $\mu_S:=S^{-1}\mu(S\otimes S)$. Note that $A_S$ is also commutative by Remark \ref{rem:equiv_and_unitary_equiv} and has the same twist of $A$ which is trivial, cf.\ Proposition \ref{prop:trivial_twist}. Therefore, $A_S$ determines a unitary CFT type VOA extension $U_S$ such that $U_S=U$ as $\C$-vector spaces thanks to \cite[Theorem 3.2]{HKL15}, cf.\ also \cite[Theorem 3.6]{HKL15} and \cite[Theorem 5.2]{KO02}, together with \cite[Proposition 1.7]{Gui19I}, Proposition \ref{prop:uniqueness_module_with_zero_weight} and \cite[Theorem 2.21]{Gui21c}. We remark that the unitary structure on $U_S$ restricts to the one of $V$, as $A$ and $A_S$ share the same $\iota$. (Note that $U_S$ is also completely unitary by \cite[Theorem 3.30]{Gui21c}.)

By \cite[Eq.\ (1.2)]{Gui21c}, see also \cite[Section 2.4]{Gui19I} and references therein, the definition of $\mu_S$ and the fact that $S\iota=\iota$ imply that $S$ is an actual VOA isomorphism from $U_S$ to $U$. As a consequence, $U$ has a unitary structure induced by the one of $U_S$ through $S$, making it a simple unitary VOA extension of $V$. Then, the complete unitarity of $U$ follows from \cite[Theorem 3.30]{Gui21c}, so proving the first claim.

The second claim follows from Theorem \ref{theo:algebras_equiv_as_Qsystems} and the discussion here above. Moreover, let $\scalar$ and $\curlyscalar$ be two unitary invariant scalar products on $U$, which restricts to the given one on $V$. The identity map from $(U,\scalar)$ to $(U,\curlyscalar)$ is a VOA automorphism, which restricts to the identity map from $(V,\scalar)$ to $(V,\curlyscalar)$ (which is also obviously unitary). Then, the former identity map is unitary and thus $\scalar$ and $\curlyscalar$ are the same, that is the unitary structure on $U$ extending the given one on $V$ is unique.
\end{proof}

\begin{cor}  \label{cor:unitary_structure_double_extension}
	Let $V$ be a regular completely unitary VOA of CFT type. Let $\widetilde{U}$ and $U$ be two simple CFT type extensions of $V$ such that $V\subset \widetilde{U} \subset U$. Then, there exists a unitary structure on $U$ making $V$ and $\widetilde{U}$ unitary subalgebras.
\end{cor}

\begin{proof}
	It is sufficient to apply Theorem \ref{theo:unitarity_voa_extensions} to $V\subset\widetilde{U}$ first, and then to $\widetilde{U}\subset U$, reminding that as an effect of the first application, $\widetilde{U}$ turns out to be regular and completely unitary.
\end{proof}

Referring e.g.\ to \cite[Section 3.4]{CKLW18}, see also \cite{Car04,CGH19,BDG22b,Lon03}, for the definition and properties of \emph{covariant subnets}, as done before for VOA extensions in Definition \ref{def:VOA extension}, we recall the following (cf.\  e.g.\  \cite{Car04,KL04}):

\begin{defin}
	An \textbf{irreducible conformal net extension} of an irreducible conformal net $\A$ is an irreducible conformal net $\B$, containing $\A$ as covariant subnet and such that the Virasoro subnet of $\B$ \emph{coincides} with the Virasoro subnet of $\A$.
\end{defin}

Indeed, we want to relate simple CFT type VOA extensions and irreducible conformal net extensions by means of the fact that both are determined by Q-systems in their respective unitary modular tensor categories of representations, see \cite{HKL15}, \cite{Gui21c} and \cite{LR95} respectively. To this end, we have to look at the core of the correspondence between unitary VOAs and conformal nets, referring to \cite{CKLW18} for details. One crucial point is that it is possible to associate to a simple unitary VOA $V$ a unique up to isomorphism irreducible conformal net $\A_V$ on the Hilbert space completion $\mathcal{H}_V$ of the vector space $V$, with respect to its scalar product, provided that certain analytic conditions on the vertex operators are satisfied. A VOA satisfying such analytic conditions is said to be \textbf{strongly local}, see \cite[Definition 6.7]{CKLW18}. 

As a further consequence of Theorem \ref{theo:unitarity_voa_extensions} and Theorem \ref{theo:haploid_rigid_iff_Q-system}, or better Theorem \ref{theo:algebras_equiv_as_Qsystems}, we obtain the following:

\begin{theo}   \label{theo:equiv_VOA_net_extensions}
	Let $V$ be a regular completely unitary VOA of CFT type. Suppose that $V$ is strongly local and that the associated irreducible conformal net $\A_V$ is completely rational. If $\Repu(V)$ is equivalent to $\Repf(\A_V)$ as a unitary modular tensor category, then there is a one-to-one correspondence between the following:
	\begin{itemize}
		\item[(i)] isomorphism classes of simple CFT type VOA extensions of $V$;
		
		\item[(ii)] (necessarily unitary) equivalence classes of normalized haploid (hence standard) commutative Q-systems in $\Repu(V)\cong \Repf(\A_V)$;
		
		\item[(iii)] isomorphism classes of irreducible conformal net extensions of $\A_V$.
	\end{itemize}
Furthermore, let $U$ be a simple CFT type VOA extension of $V$ and $\A_V^U$ be the corresponding irreducible conformal net extension of $\A_V$. Then, $U$ is a simple unitary CFT type VOA extension of $V$, it is completely unitary, and $\Repu(U)$ is equivalent to $\Repf(\A_V^U)$ as a unitary modular tensor category.
\end{theo}

\begin{proof}   

The proof of the first part goes as follows.
On the one hand, \cite[Theorem 3.2]{HKL15}, \cite[Lemma 1.20, cf.\ Theorem 5.2]{KO02}, cf.\ also \cite[Theorem 3.6 and Remark 3.7]{HKL15}, together with Remark \ref{rem:complete_unitarity}, \cite[Proposition 1.7]{Gui19I} and Proposition \ref{prop:uniqueness_module_with_zero_weight}, cf.\ also \cite[Proposition 2.1]{Gui21c}, say that isomorphism classes of simple CFT type VOA extensions of $V$ are in one-to-one correspondence with equivalence classes of rigid haploid commutative algebras in $\Repu(V)$ with trivial twist. 
By Theorem \ref{theo:algebras_equiv_as_Qsystems} and Remark \ref{rem:equiv_and_unitary_equiv}, the latter are in one-to-one correspondence with unitary equivalence classes of normalized haploid (hence standard) commutative Q-systems (with trivial twist, which is actually automatic by \cite[Theorem 3.25]{Gui21c} or Proposition \ref{prop:trivial_twist}) in $\Repu(V)$. 
On the other hand, isomorphism classes of irreducible extensions of $\A_V$ correspond bijectively to unitary equivalence classes of normalized haploid (hence standard) commutative Q-systems in $\Repf(\A_V)$ (with trivial twist by the conformal spin and statistics theorem \cite{GL96} or by Proposition \ref{prop:trivial_twist}), see \cite[Theorem 4.9]{LR95} and \cite[Lemma 15]{Lon03}, see also \cite[Section 5.2.1]{BKLR15}. Therefore, the desired one-to-one correspondence is obtained from the equivalence between $\Repu(V)$ and $\Repf(\A_V)$ as unitary modular tensor categories, see Proposition \ref{prop:1-1_equiv_algebras_under_categ_equiv} and Theorem \ref{theo:algebras_equiv_as_Qsystems}.

By Theorem \ref{theo:unitarity_voa_extensions}, $U$ is a simple unitary CFT type VOA extension of $V$, which is also completely unitary. Then, the final part of the theorem follows from the fact that $\Repu(U)$ and $\Repf(\A_V^U)$ are both equivalent to $\Mod^{\opu,0}_\mathcal{C}(A)$, respectively by \cite[Theorem 3.30]{Gui21c} and by \cite[Proposition 6.4]{BKL15}, cf.\ also \cite[Theorem 3.1]{Mue10a}, \cite[Proposition 5.1]{Bis16} and \cite[Main Theorem C]{Gui21b}. Recall that $\Mod^{\opu,0}_\mathcal{C}(A)$ (Definition \ref{def:Modu0(A)}) is the category of abstract unitary local (left) modules in $\mathcal{C} := \Repu(V) \cong \Repf(\A_V)$, over the haploid Q-system $A$, which describes the extensions $V\subset U$ and $\A_V \subset \A_V^U$. Note that $\Mod^{\opu,0}_\mathcal{C}(A)$ is denoted by $\Mod^0(A)$ in \cite{BKL15} and by $\Repu(A)$ in \cite{Gui21c}.
\end{proof}

We shall apply Theorem \ref{theo:unitarity_voa_extensions} and Theorem \ref{theo:equiv_VOA_net_extensions} to tensor products and regular cosets of some well-known VOAs listed in Table \ref{table:prop_notorious_VOAs}.
Indeed, it is known that these VOAs meet the desired hypotheses, mostly thanks to the results in \cite{Gui21a} and \cite{Gui20}. 

Before going further, it is useful to recall that for any strongly local regular completely unitary VOA $V$ of CFT type (and thus also simple), a fully faithful *-functor $\mathfrak{F}:\Repu(V)\to \Rep(\A_V)$ can be defined, assuming the \textbf{strong integrability} of $V$-modules, see \cite{CWX} or \cite[Eq. (4.4)]{Gui19II}. It is realized associating to every $V$-module $M$, a representation of $\A_V$ on the Hilbert space completion $\mathcal{H}_M$ of $M$ in a natural way, see \cite{CWX} and \cite[Theorem 4.3]{Gui19II}. Cf.\ \cite{Gui20} and references therein for a discussion of strong integrability in a more general setting and for further details. 

\begin{cor}   \label{cor:extensions_notorious_VOAs}
	Let $V$ be a tensor product of the following VOAs: unitary Virasoro VOAs with $c<1$; unitary affine VOAs associated to simple Lie algebras; discrete series $W$-algebras of type $ADE$; unitary parafermion VOAs; lattice VOAs. Let $U$ be any simple CFT type VOA extension of $V$. Then, $U$ is a simple unitary CFT type VOA extension of $V$, which is also completely unitary. Moreover, $\A_V$ is completely rational and the strong integrability *-functor $\mathfrak{F}:\Repu(V)\to \Repf(\A_V)$ is an equivalence of unitary modular tensor categories. $\mathfrak{F}$ realises a one-to-one correspondence between (isomorphism classes of) simple CFT type VOA extensions $U$ of $V$ and irreducible conformal net extensions $\A_V^U$ of $\A_V$ on the Hilbert space completion $\mathcal{H}_U$ of $U$. Furthermore, $\Repu(U)$ is equivalent to $\Repf(\A_V^U)$ as a unitary modular tensor category.  
\end{cor}

\begin{proof}
If $V$ is a tensor product of VOAs in the above list, then it is of CFT type, regular and completely unitary, see Table \ref{table:prop_notorious_VOAs}. Then, $U$ is indeed a simple unitary CFT type VOA extension of $V$, which is also completely unitary, thanks to Theorem \ref{theo:unitarity_voa_extensions}. Note that $V$ is also strongly local, see Table \ref{table:prop_notorious_VOAs}. Moreover, $\A_V$ is completely rational and  $\mathfrak{F}:\Repu(V)\to \Repf(\A_V)$ is an equivalence of unitary modular tensor categories by \cite[Theorem 5.1]{Gui21a} and \cite[Theorem I]{Gui20}. Therefore, the last part follows from Theorem \ref{theo:equiv_VOA_net_extensions}.
\end{proof}

\begin{table}  \caption{Properties of well-known classes of VOAs.} \label{table:prop_notorious_VOAs} 
	
	\begin{center}
		\resizebox{\textwidth}{!}{
			\begin{tabular}{|c|c|c|c|c|c|}
				\hline
				\rule[-4mm]{0mm}{1cm}
				\textbf{VOAs} & \textbf{Definition} & \textbf{Regularity} & \textbf{Unitarity} & \textbf{Complete Unitarity} & \textbf{Strong Locality}\\
				\hline
				\rule[-8mm]{0mm}{2cm}
				\parbox{\widthof{\textit{Virasoro VOAs}}}{\centering \textit{Unitary \\ Virasoro VOAs \\ with} $c<1$} & \parbox{\widthof{[Kac01, Ex. 4.10]}}{\cite[Sec.\ 6.1]{LL04} \\ cf.\ also:\\ \cite[Ex.\ 4.10]{Kac01}}  & \cite[Thm.\ 3.13]{DLM97}  & \cite[Thm.\ 4.2]{DL14}& \cite[Thm.\ 8.1]{Gui19II} & \cite[Ex.\ 8.4]{CKLW18} \\
				\hline
				\rule[-16mm]{0mm}{3.5cm}
				\parbox{\widthof{\textit{Positive integer}}}{\centering \textit{Positive integer\\ level \\ simple affine VOAs associated to simple Lie algebras}}  & \parbox{\widthof{[Kac01, Sec.s 5.6--5.7]}}{\cite[Sec 6.2]{LL04}\\ cf.\ also:\\ \cite[Sec.s 5.6--5.7]{Kac01}} & \cite[Thm.\ 3.7]{DLM97} & \cite[Thm.\ 4.7]{DL14} & \parbox{\widthof{[Gui19II, Thrm 8.4]}}{\cite[Thm.\ 2.7.3]{Gui20}\\ cf.\ also: \\ \cite[Thm.\ 8.4]{Gui19II} \\ \cite[Thm.\ 6.1]{Gui19a} \\ \cite[Thm.\ 5.5]{Ten19a}} & \cite[Ex.\ 8.7]{CKLW18} \\
				\hline
				\rule[-14mm]{0mm}{3cm}
				\parbox{\widthof{\textit{Positive definite}}}{\centering \textit{Positive definite \\ even lattice VOAs}}  & \parbox{\widthof{[FLM88, Thrm 8.10.2]}}{\cite[Sec.\ 6.4]{LL04} \\ cf.\ also:\\ \cite[Sec.\ 5.5]{Kac01} \\ \cite[Thm.\ 8.10.2]{FLM88}} & \cite[Thm.\ 3.16]{DLM97} & \cite[Thm.\ 4.12]{DL14} & \cite[Thm.\ 5.8]{Gui21a} & \parbox{\widthof{[Gui21a, Thrm 5.8]}}{\cite[Thm.\ 5.8]{Gui21a}\\ cf.\ also: \\ \cite[Ex.\ 8.8]{CKLW18}}\\
				\hline 
				\rule[-8mm]{0mm}{2cm}
				\parbox{\widthof{\textit{Discrete series}}}{\centering \textit{Discrete series \\ $W$-algebras \\ of $ADE$ type}} & \parbox{\widthof{[ACL19, Thrm 12.1]}}{\cite[Thm.\ 12.1]{ACL19} \\ cf.\ also:\\ \cite{KRW03}} & \parbox{\widthof{[Ara15a, Main Thrm]}}{\cite[Thm.\ 5.21]{Ara15a} \\ \cite[Main Thm.]{Ara15b}} & \parbox{\widthof{[ACL19, Thrm 12.1]}}{\cite[Thm.\ 12.1]{ACL19} \\ \centerline{+} \\ \cite[Cor.\ 2.8]{DL14}} & \parbox{\widthof{[Gui19II, Thrm 8.4]}}{\cite[Thm.\ 2.7.3]{Gui20}\\ cf.\ also: \\ \cite[Thm.\ 5.5]{Ten19a}} & \parbox{\widthof{[CKLW18, Prop. 7.8]}}{\cite[Thm.\ 12.1]{ACL19} \\ \centerline{+} \\ \cite[Prop.\ 7.8]{CKLW18}} \\
				\hline
				\rule[-8mm]{0mm}{2cm}
				\parbox{\widthof{\textit{Parafermion VOAs}}}{\centering \textit{Positive integer \\ level \\ parafermion VOAs}} & \cite{DW11} & \parbox{\widthof{[ALY14, Thrm 10.5]}}{\cite[Thm.\ 10.5]{ALY14} \\ \cite[Thm.\ 5.1]{DR17}} & \parbox{\widthof{[DW11b, Prop. 4.1]}}{\cite[Prop.\ 4.1]{DW11} \\ \centerline{+} \\ \cite[Cor.\ 2.8]{DL14}} & \cite[Thm.\ 2.7.6]{Gui20} &  \parbox{\widthof{[CKLW18, Prop. 7.8]}}{\cite[Prop.\ 4.1]{DW11} \\ \centerline{+} \\ \cite[Prop.\ 7.8]{CKLW18}} \\
				\hline
				\rule[-6mm]{0mm}{1.5cm}
				\parbox{\widthof{\textit{Tensor product}}}{\centering \textit{Tensor product\\ of VOAs}} & \parbox{\widthof{[FHL93, Sec. 2.5]}}{\cite[Sec.\ 2.5]{FHL93} \\ \cite[Sec.\ 4.3]{Kac01}} & \cite[Thm.\ 3.3]{DLM97} & \cite[Prop.\ 2.9]{DL14} & \cite[Prop.\ 3.31]{Gui21c} & \cite[Cor.\ 8.2]{CKLW18} \\
				\hline
				\rule[-11mm]{0mm}{2.5cm}
				\parbox{\widthof{\textit{(Regular) cosets}}}{\centering \textit{(Regular) cosets}} & \cite[Rem.\ 4.6b(I)]{Kac01} & \textit{(Assumed)} & \cite[Cor.\ 2.8]{DL14} & \parbox{\widthof{\cite[Theorem 2.4.1]{Gui20}}}{\cite[Thm.\ 2.6.6]{Gui20}\\ \centering{+} \\ \cite[Thm.\ 2.4.1]{Gui20}\\   \textit{(Sufficient conditions)}}  & \cite[Prop.\ 7.8]{CKLW18} \\
				\hline
			\end{tabular}
		}
	\end{center}
	\vspace{0.2cm}
	\parbox{14.5cm}{\footnotesize{Table \ref{table:prop_notorious_VOAs} gives the bibliographical references for the well-known classes of VOAs listed in the first column concerning the properties listed in the first row. Actually, the last two rows give references for two general constructions, i.e., tensor products and cosets, which can be performed within VOAs. In particular, for the coset construction, the regularity is assumed, but note that it is only used to derive the property of complete unitarity in \cite[Theorem 2.6.6]{Gui20} under some other sufficient conditions. In some cases, more than one reference is given after the text string ``cf.\ also:''. The symbol ``$+$'' means that the property ``X'' for the VOA ``Y'' is obtained by a combined action of two results. For example, the strong locality for parafermion VOAs is obtained considering the fact that these models can be constructed as coset VOAs from VOAs whose strong locality is already known.}}
\end{table}

In \cite{Hen17}, Henriques constructed VOAs  $V_{G_k}$ associated to any pair $(G,k)$ with $G$ a compact connected Lie group and  $k\in H^4_+(BG,\mathbb{Z})$.  
These VOAs where obtained as simple current extensions by finite abelian groups of tensor product of unitary affine VOAs associated to simple Lie algebras and lattice VOAs. Hence, from Corollary \ref{cor:extensions_notorious_VOAs} we get the following: 

\begin{cor}\label{cor:WZW} The WZW VOA $V_{G_k}$ is unitary for every compact connected Lie group $G$ and every $k\in H^4_+(BG,\mathbb{Z})$. 
\end{cor}

In \cite{AL17b}, Ai and Lin classified the simple CFT type VOA extensions of the simple affine VOAs $L(\mathfrak{sl}_3,k)$ associated to $\mathfrak{sl}_3$ at positive integer level $k$. 
Up to isomorphism, these extensions are given by:   $L(\mathfrak{sl}_3,k)$ for any positive integer $k$, a $\mathbb{Z}_3$-simple current extension of $L(\mathfrak{sl}_3,k)$
for all $k\equiv 0\; (\mathrm{mod} 3)$,  the conformal inclusions $L(\mathfrak{sl}_3,5) \subset L(\mathfrak{sl}_6,1)$, $L(\mathfrak{sl}_3,9) \subset L(E_6,1)$
and $L(\mathfrak{sl}_3,21) \subset L(E_8,1)$. From Corollary \ref{cor:extensions_notorious_VOAs} we get the following:

\begin{cor} \label{cor:sl3} The irreducible extensions of the loop group  conformal nets $\mathcal{A}_{L(\mathfrak{sl}_3,k)}$ are given by: $\mathcal{A}_{L(\mathfrak{sl}_3,k)}$ for any positive integer $k$, a $\mathbb{Z}_3$-simple current extension of $\A_{L(\mathfrak{sl}_3,k)}$ for all $k\equiv 0\; (\mathrm{mod} \, 3)$,  the conformal inclusions $\A_{L(\mathfrak{sl}_3,5)} \subset \A_{L(\mathfrak{sl}_6,1)}$, 
$\A_{L(\mathfrak{sl}_3,9)} \subset \A_{L(E_6,1)}$ and $\A_{L(\mathfrak{sl}_3,21)} \subset \A_{L(E_8,1)}$.
\end{cor}

By \cite[Definition 2.1 and Lemma 2.2]{DGH98}, a \textbf{framed VOA} with $c=\frac{n}{2}$ with $n\in\Zplus$, is a simple CFT type VOA extension of a tensor product of $n$ copies of the Virasoro VOA with central  charge $\frac{1}{2}$. Similarly, a \textbf{framed conformal net} \cite[Definition 4.1]{KL06} with $c=\frac{n}{2}$ is an irreducible conformal net extension of a tensor product of $n$ copies of the Virasoro conformal net with central  charge $\frac{1}{2}$. Consequently, for framed VOAs and framed conformal nets, Corollary \ref{cor:extensions_notorious_VOAs} implies the following:

\begin{cor}  \label{cor:framed_chiral_CFTs}
	Let $n\in\Zplus$. Then, there is a one-to-one correspondence between framed VOAs $V$ with central  charge $c=\frac{n}{2}$ and framed conformal nets $\A$ with the same central charge $c$. Furthermore, $V$ is completely unitary, $\A$ is completely rational, and $\Repu(V)$ is equivalent to $\Repf(\A)$ as a unitary modular tensor category. 
\end{cor}

Note that the discrete series $W$-algebras, as well as the parafermion VOAs, can be obtained through a regular coset construction of affine and lattice VOAs, see Table \ref{table:prop_notorious_VOAs}. Therefore, for these models, the results stated in Corollary \ref{cor:extensions_notorious_VOAs} are just an application of the following more general statement:

\begin{cor}  \label{cor:extensions_coset}
	Let $V$ and $U$ be tensor products of the following VOAs: unitary Virasoro VOAs with $c<1$; unitary affine VOAs associated to simple Lie algebras; discrete series $W$-algebras of type $ADE$; unitary parafermion VOAs; lattice VOAs. Suppose that $V$ is a unitary subalgebra of $U$ and that the coset VOA  $V^c$ of $V$ in $U$ is regular. Let $\widetilde{U}$ be any simple CFT type VOA extension of $V^c$. Then, $\widetilde{U}$ is a simple unitary CFT type VOA extension of $V^c$, which is also completely unitary. Moreover, the coset net $\A_V^c$ (which is equal to $\A_{V^c}$ by \cite[Proposition 7.8]{CKLW18}) is completely rational and the strong integrability *-functor $\mathfrak{F}:\Repu(V^c)\to \Repf(\A_V^c)$ is an equivalence of unitary modular tensor categories. $\mathfrak{F}$ realises a one-to-one correspondence between (isomorphism classes of) simple CFT type VOA extensions $\widetilde{U}$ of $V^c$ and irreducible conformal net extensions $\A_{V^c}^{\widetilde{U}}$ of $\A_V^c$ on the Hilbert space completion $\mathcal{H}_{\widetilde{U}}$ of $\widetilde{U}$.  Furthermore, $\Repu(\widetilde{U})$ is equivalent to $\Repf(\A_{V^c}^{\widetilde{U}})$ as a unitary modular tensor category.  
\end{cor}

\begin{proof}
The coset VOA $V^c$ is of CFT type and completely unitary by \cite[Theorem I]{Gui20}, see also the last row of Table \ref{table:prop_notorious_VOAs}. Then, $\widetilde{U}$ is indeed a simple unitary CFT type VOA extension of $V^c$, which is also completely unitary, thanks to Theorem \ref{theo:unitarity_voa_extensions}. Note that $V$, $U$ and $V^c$ are also strongly local, see Table \ref{table:prop_notorious_VOAs}. Furthermore, $\A_{V^c}=\A_V^c$ is completely rational and  $\mathfrak{F}:\Repu(V^c)\to \Repf(\A_V^c)$ is an equivalence of unitary modular tensor categories by \cite[Theorem I]{Gui20}. Then, the last part follows from Theorem \ref{theo:equiv_VOA_net_extensions}.
\end{proof}

\subsection{Applications to VOSAs and graded-local conformal nets}

Now, we draw some consequences of Theorem \ref{theo:unitarity_voa_extensions} and Theorem \ref{theo:equiv_VOA_net_extensions} in the VOSA context. 

The definition of vertex subalgebra in the VOSA setting  does essentially not differ from the one in the VOA one, see e.g.\ \cite[Section 4.3]{Kac01}. For the definition of  unitary (vertex) subalgebra, see e.g.\ \cite{CGH} and \cite[Section 2.6]{Gau21}. 
In this paper, we will only consider VOSAs with \emph{correct statistics}, cf.\ \cite{CKL19}. Hence, we always assume that  the 
$\frac12\mathbb{Z}$-grading of a VOSA $V = V_\parzero \oplus V_\parone$ satisfies 
$$V_\parzero = \bigoplus_{n \in \mathbb{Z}}V_n \quad \textrm{and} \quad V_\parone = \bigoplus_{n \in  \mathbb{Z} + \frac12}V_n \,. $$ 
Hence, $V_\parzero$ is a VOA while $V_\parone$ is a  $V_\parzero$-module with twist $e^{i2\pi L_0}\restriction_{V_\parone} 
= -\one_{V_\parone}$. 

\begin{defin}
	A simple VOSA $U$ of CFT type is called a \textbf{simple CFT type VOSA extension} of a simple VOSA $V$ of CFT type if the latter, namely $V$, is a vertex subalgebra of $U$ with the same conformal vector. Moreover, $U$ is said in addition to be a simple \textbf{unitary} CFT type VOSA extension of $V$ if both VOSAs are unitary and if the unitary structure of $U$ restricts to the one of $V$, i.e., $V$ is a unitary subalgebra of $U$.
\end{defin}

\begin{rem} \label{rem:simplicity_even_part}
It is well-known that if $V$ is a simple VOSA, then its even part $V_\parzero$ is a simple VOA. Indeed, if $\mathscr{I}$ is an ideal of $V_\parzero$, then the linear span $V\cdot \mathscr{I}$ of $\{a_{(n)}b\mid a\in V, b\in \mathscr{I}, n\in\Z \}$ is an ideal of $V$, see e.g.\ \cite[Proposition 4.5.6]{LL04} and \cite[Lemma 6.1.1]{Li94}, cf.\ \cite[Lemma 2.7.1]{Gau21}. For a similar reason, the odd part $V_\parone$ is an irreducible $V_\parzero$-module. Note also that, if $V$ is a VOSA of CFT type, then $V_\parzero$ is of CFT type too.
\end{rem}

From Theorem \ref{theo:unitarity_voa_extensions}, in the VOSA context, we get:

\begin{cor} \label{cor:unitarity_vosa_extensions}
	Let $V$ be a regular completely unitary VOA of CFT type and $U$ be a simple CFT type VOSA extension of $V$. Then, $U$ is a simple unitary CFT type VOSA extension of $V$.
\end{cor}

\begin{proof}
The even part $U_\parzero$ of $U$ is a simple CFT type VOA extension of $V$, whereas the odd part $U_\parone$ is an irreducible 
$U_\parzero$-module, see Remark \ref{rem:simplicity_even_part}. Then, $U_\parzero$ is a simple unitary CFT type VOA extension of $V$ which is also completely unitary by Theorem \ref{theo:unitarity_voa_extensions}. Thus, $U_\parone$  is a unitary $U_\parzero$-module.
In conclusion, the unitarity of $U$, as a simple CFT type VOSA extension of $V$, follows by a (kind of) ``super'' version of \cite[Theorem 3.3]{DL14} as given in \cite{CGH}, see \cite[Corollary 2.7.3]{Gau21}.
\end{proof}

We now recall the notion of graded-local conformal net extension, cf.\  \cite{CKL08,CHKLX15,CHL15}.
For the related notion of covariant subnet of a graded-local conformal net we refer the reader to \cite{CKL08,CHKLX15,CHL15}, see also
\cite{CGH} and \cite[Section 1.3]{Gau21}. 

\begin{defin}
	An \textbf{irreducible graded-local conformal net extension} of an irreducible graded-local conformal net $\A$ is an irreducible graded-local conformal net $\B$, containing $\A$ as covariant subnet and such that the Virasoro subnet of $\B$ coincides with the Virasoro subnet of $\A$. 
\end{defin}

\begin{rem}\label{rem:graded-local extension}
Let  $\A_\parzero$ and $\B_\parzero$ be the even (or Bose) subnets of $\A$ and $\B$ respectively. Assume that $\A$ is a covariant subnet of $\B$. Then $\A_\parzero$ is a covariant subnet of $\B_\parzero$ and $\B$ is an irreducible graded-local conformal net extension of $\A$ if and only if $\B_\parzero$ is an irreducible conformal net extension of $\A_\parzero$. 
\end{rem}

From Theorem \ref{theo:equiv_VOA_net_extensions}, in the VOSA context, we get the following Theorem \ref{theo:equiv_VOSA_net_extensions}. In order to prove it, recall the definition of a $\Z_2$-simple current in a ($C^*$-)tensor category $\mathcal{C}$ (also called self-dual simple current), i.e., an invertible object $J$ in $\mathcal{C}$ not isomorphic to 
$\id_{\mathcal{C}}$ and such that $J \cong \overline{J}$, cf.\ Remark \ref{rem:twistinverible}.

\begin{theo}  \label{theo:equiv_VOSA_net_extensions}
	Let $V$ be a regular completely unitary VOA of CFT type. Suppose that $V$ is strongly local and that the associated irreducible conformal net $\A_V$ is completely rational. If $\Repu(V)$ is equivalent to $\Repf(\A_V)$ as a unitary  modular tensor category, then there is a one-to-one correspondence between the following:
		\begin{itemize}
			\item[(i)] isomorphism classes of simple CFT type VOSA extensions of $V$;
		
			\item[(ii)] isomorphism classes of irreducible graded-local conformal net extensions of $\A_V$.
		\end{itemize}
\end{theo}

\begin{proof}
Let $U$ be a simple CFT type VOSA extension	of $V$. The even part $U_\parzero$ of $U$ is simple as CFT type VOA extension of $V$, whereas $U_\parone$ is an irreducible $U_\parzero$-module, see Remark \ref{rem:simplicity_even_part}. By Theorem \ref{theo:unitarity_voa_extensions},
$U_\parzero$ is a simple unitary CFT type VOA extension	of $V$, which is also completely unitary so that $U_\parone$ is a unitary $U_\parzero -$module.
Moreover, there exists an irreducible conformal net $\A_V^{U_\parzero}$, extending $\A_V$ and such that $\Repu(U_\parzero)$ is equivalent to $\Repf(\A_V^{U_\parzero})$ as a unitary modular tensor category.
Now, the odd part $U_\parone$ of $U$ is a $\Z_2$-simple current $J$ in $\Repu(U_\parzero)$, see e.g.\ \cite[Theorem 3.1 and Remark A.2]{CKLR19} with twist $\omega(J)=-\one_J$ because we are assuming the correct statistics for $U$. 

By the equivalence of $\Repu(U_\parzero)$ and $\Repf(\A_V^{U_\parzero})$, $J$ determines a  $\mathbb{Z}_2$-simple current 
$\tilde{J}$ in $\Repf(\A_V^{U_\parzero})$ with twist  $\omega(\tilde{J}) = -\one_{\tilde{J}}$, and hence an irreducible graded-local conformal net extension of $\A_V^{U_\parzero}$, see e.g.\ \cite{CKL08,CHKLX15}. This shows that any simple CFT type VOSA extension of $V$ gives rise to an irreducible graded-local conformal net extension of $\A_V$. 

Conversely, by  Theorem \ref{theo:equiv_VOA_net_extensions}, the even part of any irreducible graded-local conformal net extension of $\A_V$  is of the form $\A_V^{U_\parzero}$ for some simple CFT type VOA extension $U_\parzero$ of $V$. Accordingly, any irreducible graded-local conformal net extension of $\A_V$ 
comes from a  $\mathbb{Z}_2$-simple current $\tilde{J}$ in $\Repf(\A_V^{U_\parzero})$, with twist $\omega(\tilde{J}) = -\one_{\tilde{J}}$,  see again \cite{CKL08,CHKLX15}.  By the equivalence of $\Repu(U_\parzero)$ and $\Repf(\A_V^{U_\parzero})$, $\tilde{J}$ determines a  
$\mathbb{Z}_2$-simple current $J$ in $\Repu(U_\parzero)$ with twist  $\omega(\tilde{J}) = -\one_{\tilde{J}}$.
Thus, the braiding satisfies $b_{J,J} =-\one_{J\otimes J}$ by Remark \ref{rem:twistinverible}, since $\Repu(U_\parzero)$ is a unitary modular tensor category. Then, $U:= U_\parzero \oplus J$ admits a unique structure of simple VOSA compatible with its $U_\parzero$-module structure, see \cite[Theorem 3.9]{CKL19} and 
\cite[Proposition 5.3]{DM04b} and $U$ is CFT type because $J$ is unitary and hence non-negatively graded.
\end{proof}

Following Theorem \ref{theo:equiv_VOSA_net_extensions}, for a simple CFT type VOSA extension $U$ of the VOA $V$, we denote by $\A_V^U$ the corresponding irreducible graded-local conformal net extension of $\A_V$. 
We present our main result in the VOSA context:

\begin{theo}  \label{theo:equiv_VOSA_net_super_extensions}
	Let $V$ be a simple VOSA of CFT type and $\A$ be an irreducible graded-local conformal net. Assume that $V_\parzero$ is completely unitary and strongly local, and that $\A_\parzero$ is completely rational. Assume also that $\A_{V_\parzero}$ is isomorphic to $\A_\parzero$ and that $\Repu(V_\parzero)$ and $\Repf(\A_\parzero)$ are equivalent as unitary  modular tensor categories. 
	Suppose that the equivalence class of the $\Z_2$-simple current defining $V$ in $\Repu(V_\parzero)$ is mapped into the equivalence class of the one defining $\A$ in $\Repf(\A_\parzero)$ by the modular tensor equivalence between $\Repu(V_\parzero)$ and $\Repf(\A_\parzero)$ (see the proof of Theorem \ref{theo:equiv_VOSA_net_extensions}). Then, every simple CFT type VOSA extension $U$ of $V_\parzero$ extends also $V$ if and only if the corresponding irreducible graded-local conformal net extension $\A_{V_\parzero}^U$ of $\A_{V_\parzero}\cong\A_\parzero$ extends $\A$. In other words, there is a one-to-one correspondence between (isomorphism classes of) simple CFT type VOSA extensions of $V$ and irreducible graded-local conformal net extensions of $\A$. 
	\end{theo}

The proof of Theorem \ref{theo:equiv_VOSA_net_super_extensions} makes use of the following two propositions.

For the first one, consider a simple VOSA $V$ of CFT type and a simple CFT type VOSA extension $U$ of $V_\parzero$. Note that $V_\parzero$ and $U_\parzero$ are simple and of CFT type, whereas $V_\parone$ and $U_\parone$ are an irreducible $V_\parzero$-module  and $U_\parzero$-module respectively, see Remark \ref{rem:simplicity_even_part}. 
Suppose further that: $V_\parzero$ is regular and self-dual (then, it is strongly rational); non-zero vectors of the irreducible $V_\parzero$-modules have non-negative real conformal weights; $V_\parzero$ is the only irreducible $V_\parzero$-module with a non-zero vector of conformal weight zero. (Note that these two last conditions are satisfies if $V_\parzero$ is strongly unitary by \cite[Proposition 1.7]{Gui19I} and Proposition \ref{prop:uniqueness_module_with_zero_weight}.)
From \cite[Theorem 3.6]{HKL15}, cf.\ \cite[Theorem 5.2]{KO02}, $\Rep(U_\parzero)$ is equivalent to $\Mod_\mathcal{C}^0(B_\parzero)$ (Definition \ref{def:Modu0(A)}) as a modular tensor category, where $\mathcal{C}:=\Rep(V_\parzero)$ and $B_\parzero$ is the haploid commutative algebra in $\mathcal{C}$ with trivial twist, realizing the extension $V_\parzero\subset U_\parzero$. By \cite[Lemma 1.20]{KO02}, $B_\parzero$ is also rigid.
Denote by $A_\parone$ the $\Z_2$-simple current in $\mathcal{C}$ which uniquely determines $V$, as explained in the proof of Theorem \ref{theo:equiv_VOSA_net_extensions}. It corresponds to the $V_{\parzero}$-module $V_\parone$. Similarly, $B_\parone$ is the $\Z_2$-simple current in $\Rep(U_\parzero)\cong \Mod_\mathcal{C}^0(B_\parzero)$ giving $U$.
With an abuse of notation, we identify $B_\parzero$ with its corresponding object and we write $m:B_\parzero\otimes B_\parzero\to B_\parzero$ for its multiplication as algebra in $\mathcal{C}$.
Accordingly, define the object $(B_\parzero\otimes A_\parone, m_{B_\parzero\otimes A_\parone})$ in $\Mod_\mathcal{C}(B_\parzero)$ by $ m_{B_\parzero\otimes A_\parone}:= (m\otimes \one_{A_\parone})a_{B_\parzero, B_\parzero, A_\parone}$.
Then:

\begin{prop}  \label{prop:charact_vosa_extensions}
Let $V$ be a simple VOSA of CFT type with $V_\parone \neq \{0\}$. 
Assume that: $V_\parzero$ is strongly rational; non-zero vectors of the irreducible $V_\parzero$-modules have non-negative real conformal weights; $V_\parzero$ is the only irreducible $V_\parzero$-module with a non-zero vector of conformal weight zero. Then, $U$ is a simple CFT type VOSA extension of $V$ if and only if $B_\parzero\otimes A_\parone$ is a (local) left $B_\parzero$-module equivalent to $B_\parone$.
\end{prop}

\begin{proof} 
Suppose that $U$ is a simple CFT type extension of $V$. By \cite[Theorem 1.6 point 2.]{KO02}, cf.\ \cite[Lemma 2.61]{CKM21}, we have that $\Hom_{\Mod_\mathcal{C}(B_\parzero)}(B_\parzero\otimes A_\parone, B_\parone)$ is naturally isomorphic to $\Hom_\mathcal{C}(A_\parone, B_\parone)$. 
On the one hand, as $A_\parone$ is invertible, $B_\parzero\otimes A_\parone$ is an irreducible $B_\parzero$-module. 
On the other hand, as $U$ extends $V$, $A_\parone$ is a $V_\parzero$-submodule of $B_\parone$ and thus $\Hom_\mathcal{C}(A_\parone, B_\parone)$ is non-zero. Therefore, $\Hom_{\Mod_\mathcal{C}(B_\parzero)}(B_\parzero\otimes A_\parone, B_\parone)$ is non-zero too. Thus, $B_\parzero\otimes A_\parone$ and $B_\parone$ must be equivalent left $B_\parzero$-modules. In particular, $B_\parzero\otimes A_\parone$ is local as $B_\parone$ is.

Vice versa, suppose that $B_\parzero\otimes A_\parone$ and $B_\parone$ are equivalent (local) left $B_\parzero$-modules. 
Then, they are equivalent as $V_\parzero$-module too. As $A_\parone$ is a $V_\parzero$-submodule of $B_\parzero\otimes A_\parone$, it is a $V_\parzero$-submodule of $B_\parone$ and thus of $U$.
\footnote{With an abuse of notation, we denote $V_\parzero$ and $U$, seen as $V_\parzero$-modules, still with their respective VOSA symbols, leaving their interpretation as VOSAs or as $V_\parzero$-modules to the context.}
Now, consider the $V_\parzero$-submodule of $U$ given by the direct sum $V_\parzero\oplus A_\parone$. Our goal is to prove that $V_\parzero\oplus A_\parone$
is a simple vertex subalgebra of $U$.

To this end, let $a,b\in V_\parzero$ and $c,d\in A_\parone$ and consider their modes $a_{(n)}, b_{(m)}, c_{(h)}, d_{(k)}$ for arbitrary $n,m,h,k\in\Z$, given by the state-field correspondence of $U$. $a_{(n)}b\in V_\parzero$ as $V_\parzero$ is a vertex subalgebra of $U$. $a_{(n)}c\in A_\parone$ as $A_\parone$ is a $V_\parzero$-submodule of $U$. By the skew-symmetry \cite[Section 4.2]{Kac01} for $U$, it is easy to check that $c_{(k)}a\in A_\parone$. Then, it remains to prove that $c_{(k)}d\in V_\parzero$. For any pair of subsets 
$C$ and $D$ of $U$ we denote by $C\cdot D$ the vector subspace of $U$ linearly spanned by vectors $c_{(k)}d$ for arbitrary $c \in C$, $d\in A_\parone$ and $k\in \Z$. Using the Borcherds commutator formula \cite[Section 4.8]{Kac01} for $U$, we can prove that $A_\parone\cdot A_\parone$ is a $V_\parzero$-module. 
Recall that $A_\parone$ is a proper $V_\parzero$-submodule of $B_\parone$ and $B_\parone\otimes B_\parone=B_\parzero$ as $B_\parzero$-modules and thus as $V_\parzero$-modules too.
Then, $A_\parone\cdot A_\parone$ is a proper $V_\parzero$-submodule of $B_\parzero$. 
Now, let $L$ be any irreducible $V_\parzero$-submodule of $A_\parone\cdot A_\parone$ different from $V_\parzero$. Since $U$ is of CFT type, $L$ must be inequivalent to $V_\parzero$. If $Y_U$ is the state-field correspondence for $U$, then $Y_U(\cdot,z)|_{A_\parone}$ is an intertwining operator for $V_\parzero$-modules of type $\binom{A_\parone\cdot A_\parone}{A_\parone \,\, A_\parone}$ (obviously restricting the charge space to $A_\parone$).
Consider a projection $E_L$ from $U$ onto $L$ commuting with the $V_\parzero$ action on $U$. Then, $E_L Y_U(\cdot,z)|_{A_\parone}$ is an intertwining operator for $V_\parzero$-modules of type $\binom{L}{A_\parone \,\, A_\parone}$ (still restricting the charge space to $A_\parone$). Then, $E_L Y(\cdot,z)|_{A_\parone}$ must be zero as $A_\parone$ is a $\Z_2$-simple current in $\mathcal{C}$. It follows that $L=E_{L}A_\parone\cdot A_\parone=\{0\}$ and thus $A_\parone\cdot A_\parone \subset V_\parzero$. This implies that $V_\parzero \oplus A_\parone$ is a vertex subalgebra of $U$. It remains to prove that this subalgebra is simple.  

To this end we first prove that 
$A_\parone \cdot A_\parone \neq \{ 0 \}$ so that $A_\parone \cdot A_\parone = V_\parzero$ because $A_\parone \cdot A_\parone$ is an ideal of the simple VOA $V_\parzero$. Assume by contradiction that  $A_\parone \cdot A_\parone = \{ 0 \}$. It follows that  
$Y(a,z)Y(b,w)c = 0$ for all $a \in U$ and all $b, c \in A_\parone$. Hence, for any $a \in U$ and any $b, c \in A_\parone$, there is a 
positive integer $N$ such that $(z-w)^NY(b,z)Y(a,w)c = 0$. Now, if $d$ is in the restricted dual $U'$ of $U$ then 
$\langle d, Y(b,z)Y(a,w)c \rangle$ converges to a rational function $R(z,w)$ in the domain $|z| > |w|$, see e.g.\  \cite[Proposition 3.2.1]{FHL93} and \cite[Remark 4.9a]{Kac01}.  Since $(z-w)^NR(z,w) = 0$ for $|z| > |w|$ we must have  $R(z,w)=0$ and thus $\langle d, Y(b,z)Y(a,w)c \rangle =0$. Since $d \in U'$ was arbitrary we can conclude that $Y(b,z)Y(a,w)c =0$. Now, 
$U_\parzero \cdot A_\parone$ is  a  non-zero $U_\parzero$-submodule of $U$  by \cite[Proposition 4.5.6]{LL04} and thus it must coincide with $U_\parone $ because $U$ is a simple VOSA. It follows that  $A_\parone \cdot U_\parone=\{0\}$ and hence 
 $U_\parone \cdot A_\parone=\{0\}$ by skew-symmetry. Hence, by the Borcherds commutator formula we find  that
 $U_\parone \cdot U_\parone  = U_\parone \cdot \left( U_\parzero \cdot A_\parone \right) =\{ 0 \}$ so that $U \cdot U_\parone = U_\parone$ in contradiction with the simplicity of $U$.  Therefore, $A_\parone \cdot A_\parone = V_\parzero$. 
  
Now, let $\mathscr{I}$ be a non-zero VOSA ideal of $V_\parzero \oplus A_\parone$. Then, $\mathscr{I}$ is a non-zero 
$V_\parzero$-submodule of $V_\parzero \oplus A_\parone$ and hence $V_\parzero \subset \mathscr{I}$ or $A_\parone \subset \mathscr{I}$ . 
If $V_\parzero \subset \mathscr{I}$, then also $A_\parone \subset \mathscr{I}$ so that $\mathscr{I}= V_\parzero \oplus A_\parone$. 
If $A_\parone \subset \mathscr{I}$, then $V_\parzero =  A_\parone \cdot A_\parone \subset \mathscr{I}$ and again 
$\mathscr{I} = V_\parzero \oplus A_\parone$. Hence, $V_\parzero \oplus A_\parone$ is a simple VOSA.

To sum up, we have proved that $V_\parzero\oplus A_\parone$ is a $\Z_2$-simple current extension of $V_\parzero$, giving a simple vertex subalgebra of $U$. It is well known that the VOSA structure of a simple current extension is unique, up to isomorphism, cf.\ \cite[Proposition 5.3]{DM04b}. Therefore, $V_\parzero\oplus A_\parone$ must be isomorphic to $V$, concluding the proof.
\end{proof}

For the second proposition, we proceed similarly to the first one. 
Let $\A$ be an irreducible graded-local conformal net and $\B$ be an irreducible graded-local conformal net extension of $\A_\parzero$. Suppose that $\A_\parzero$ is completely rational and that it is a proper subnet of $\A$.
By \cite[Proposition 6.4]{BKL15}, cf.\ also \cite[Theorem 3.1]{Mue10a}, \cite[Proposition 5.1]{Bis16} and \cite[Main Theorem C]{Gui21b}, 
$\Repf(\B_\parzero)$ is equivalent to $\Mod_\mathcal{C}^{\opu,0}(B_\parzero)$ as a unitary modular tensor category, where $\mathcal{C}:=\Repf(\A_\parzero)$ and $B_\parzero$ is the normalized haploid commutative Q-system in $\mathcal{C}$ (with trivial twist by \cite{GL96} or by Proposition \ref{prop:trivial_twist}), realizing the extension $\A_\parzero\subset \B_\parzero$.
Call $A_\parone$ the $\Z_2$-simple current in $\mathcal{C}$ which uniquely determines $\A$, as explained in the proof of Theorem \ref{theo:equiv_VOSA_net_extensions}. Similarly, $B_\parone$ is the $\Z_2$-simple current in $\Repf(\B_\parzero)\cong \Mod_\mathcal{C}^{\opu,0}(B_\parzero)$ giving $U$.
Also in this case, identify, with an abuse of notation, $B_\parzero$ with its corresponding object and say $m:B_\parzero\otimes B_\parzero\to B_\parzero$ for its multiplication as algebra in $\mathcal{C}$.
Similarly, define the object $(B_\parzero\otimes A_\parone, m_{B_\parzero\otimes A_\parone})$ in $\Mod_\mathcal{C}^{\opu}(B_\parzero)$ by $ m_{B_\parzero\otimes A_\parone}:= (m\otimes \one_{A_\parone})a_{B_\parzero, B_\parzero, A_\parone}$.
Therefore: 

\begin{prop}  \label{prop:charact_graded_net_extensions} Let $\A$ be an irreducible graded-local conformal net and assume that 
$\A_\parzero$ is completely rational. Then, $\B$ is an irreducible graded-local conformal net extension of $\A$ if and only if $B_\parzero\otimes A_\parone$ is a (necessarily local) left $B_\parzero$-module equivalent to $B_\parone$.
\end{prop}

\begin{proof}
It is not difficult to see that $\B$ extends $\A$ if and only if $B_\parone$ is isomorphic to the \textit{$\alpha$-induction} 
$\alpha_{A_\parone}$ of $A_\parone$, for the 
definition and properties of $\alpha$-induction see  \cite{BE98,LR95}. 

Now, $\alpha_{A_\parone}$ is an automorphism (possibly solitonic) of the net $\B_\parzero$ by \cite[Theorem 3.21]{BE98}, and thus it is an irreducible endomorphism. Similarly,
$B_\parzero \otimes A_\parone$ is invertible and hence irreducible in ${\Mod_{\mathcal C}^{\opu}(B_\parzero)}$ by \cite[Theorem 2.59]{CKM21}. 

If $B_\parzero\otimes A_\parone$ is a left $B_\parzero$-module equivalent to $B_\parone$, then we can identify 
$B_\parzero\otimes A_\parone$ with a DHR automorphism of the net $\B_\parzero$ by \cite[Proposition 6.4]{BKL15}. Then, it follows by the  
\textit{$\alpha \sigma$-reciprocity} in \cite[Theorem 3.21]{BE98}  that $\alpha_{A_\parone}$  and $B_\parzero\otimes A_\parone$ are locally equivalent endomorphisms of $\B_\parzero$. Thus, they are globally equivalent by  strong additivity. It follows that 
$\alpha_{A_\parone}$ is a DHR automorphism of $\B_\parzero$ equivalent to $B_\parone$ and hence $\B$ is an irreducible graded-local conformal net extension of $\A$. 

Conversely, if  $\B$ is an irreducible graded-local conformal net extension of $\A$ then $\alpha_{A_\parone}$ is a DHR automorphism 
of $\B_\parzero$ equivalent to $B_\parone$. Now, $\alpha_{A_\parone}$ restricts to a DHR representation of $\A_\parzero$ equivalent 
to $B_\parzero\otimes A_\parone$, cf.\ \cite[Proposition 3.9]{LR95}. It follows that the twist $\omega(B_\parzero\otimes A_\parone)$ is equal to $-\one_{B_\parzero} \otimes \one_{A_\parone} = \omega(B_\parzero) \otimes \omega(A_\parone)$. As a consequence,
the monodromy $\mathcal{M}_{B_\parzero, A_\parone}:=b_{A_\parone, B_\parzero} b_{B_\parzero, A_\parone}$ is equal to $\one_{B_\parzero\otimes A_\parone}$. Thus,  $B_\parzero\otimes A_\parone$  is local and using again the \textit{$\alpha \sigma$-reciprocity} in \cite[Theorem 3.21]{BE98}, we find that $B_\parzero\otimes A_\parone$ is a left $B_\parzero$-module equivalent to $B_\parone$.
\end{proof}

\medskip

\begin{proof}[Proof of Theorem \ref{theo:equiv_VOSA_net_super_extensions}]
By Theorem \ref{theo:equiv_VOSA_net_extensions} and its proof, there is a one-to-one correspondence between simple CFT type VOSA extensions of $V_\parzero$ and irreducible graded-local conformal net extensions of $\A_{V_\parzero}\cong \A_\parzero$. Moreover, if $U$ is such an extension of $V_\parzero$, then $U_\parzero$ is completely unitary by Theorem \ref{theo:equiv_VOA_net_extensions} and $\Repu(U_\parzero)\cong \Mod_{\mathcal{C}}^{\opu, 0}(B_\parzero) \cong \Repf(\A_\parzero^{U_\parzero})$ are equivalences of unitary modular tensor categories, where $B_\parzero$ is the normalized haploid commutative Q-system (with trivial twist by \cite{GL96} or by Proposition \ref{prop:trivial_twist}) in $\mathcal{C}:= \Repu(V_\parzero) \cong \Repf(\mathcal{A}_\parzero)$, giving the extensions $V_\parzero\subset U_\parzero$ and $\A_\parzero\subset \A_\parzero^{U_\parzero}$. Note that, without loss of consistency, we can consider every $\Z_2$-simple current in $\Rep(U_\parzero)$ as an object in $\Repu(U_\parzero)$, see Remark \ref{rem:complete_unitarity}. Then, the statement of the theorem follows from Proposition \ref{prop:charact_vosa_extensions} and Proposition \ref{prop:charact_graded_net_extensions}.
\end{proof}

\section{Unitarity and strong locality of the Schellekens list}   \label{sec:unitarity_schellekens_list}

We start with two central definitions, see Section \ref{sec:applications} for the general setting.

\begin{defin}
	A strongly rational \textbf{holomorphic} VOA is a strongly rational VOA, which has itself as unique irreducible module.
\end{defin}

Similarly, we have:

\begin{defin}
	A \textbf{holomorphic} conformal net is a completely rational conformal net, which has itself as unique irreducible representation, or equivalently \cite[Definition 2.6]{KL06}, which has $\mu$-index equal to 1. 
\end{defin}	

Over the recent years, it was proved that any strongly rational holomorphic VOA with central  charge $c=24$  is completely determined by the Lie algebra structure on its weight-one subspace if the latter is non-trivial. 
More precisely, every strongly rational holomorphic VOA $U$ with central charge $c=24$ and with $U_1\not=\{0\}$ is a simple CFT type VOA extension of its affine vertex subalgebra $V_{U_1}$.
Moreover, the Lie algebra $U_1$ must be in the list of complex Lie algebras in Schellekens' work \cite{Sch93} on holomorphic $c=24$ CFTs, see \cite{DM04a}, \cite{EMS20} and \cite{ELMS21}. This list \cite[Table 1]{Sch93} counts 71 entries in total. To entry  0 is associated the \textbf{Moonshine VOA}, see e.g.\ \cite{FLM88}, which has indeed trivial weight-one subspace. Although the uniqueness of the Moonshine VOA is still a very important open problem, the existence and uniqueness of the other 70 VOAs was proved in a series of papers, based on a case-by-case analysis, see e.g.\ \cite{LS19} (and \cite{LS20}) for a review, and through uniform approaches slightly later, see \cite{Hoh17} and \cite{Lam19,MS20,MS21,HM20,CLM22,BLS22}, cf.\ also \cite{LM22}.

Regarding the unitarity of these 71 VOAs, it is already known for the 24 of them arising from the corresponding \textit{Niemeier lattices}, see e.g.\ \cite{CS99} for these lattices, see \cite[Figure 6.1]{Mol16} for the corresponding entries in the list, see the third row of Table \ref{table:prop_notorious_VOAs} for lattice VOAs. The \textbf{Leech lattice} $\Lambda$ and the corresponding \textbf{Leech lattice VOA} $V_\Lambda$ are among them.  
Furthermore, the unitarity of the Moonshine VOA was established in \cite[Theorem 4.15]{DL14} and in \cite[Example 5.10]{CKLW18}. 

In the first result here below (Theorem \ref{theo:unitarity_holomorphicVOAs}), we apply what has been developed in Section \ref{sec:applications} to give a uniform proof of the \emph{unitarity} of the 69 VOAs from the Schellekens list, excluding the Leech lattice VOA and the Moonshine VOA, but including all the previously missing cases. For each of these 69 entries, we show the \emph{existence} of a holomorphic conformal net with central  charge $c=24$, as an extension of the affine conformal net associated with the corresponding strongly local unitary affine VOA, see the second row of Table \ref{table:prop_notorious_VOAs}. We point out that some of these conformal nets are new, see Remark \ref{rem:other_holomorphic_conformal_net_constructions} 
and Remark \ref{rem:other_holomorphic_conformal_net_constructions2} for details. We also prove a \emph{uniqueness} result for all of them. A similar uniqueness result holds for the Leech lattice conformal net, which will be treated separately for expository reasons (Theorem \ref{theo:leech_uniqueness}), just after the following:

\begin{theo}  \label{theo:unitarity_holomorphicVOAs}
	Every strongly rational holomorphic VOA $U$ with central  charge $c=24$ and weight-one subspace $U_1$ being a semisimple Lie algebra is unitary. Let $V_{U_1}$ be the affine vertex subalgebra of $U$ generated by $U_1$ and let $\A_{V_{U_1}}$ be the irreducible conformal net associated to $V_{U_1}$. Then, there exists a  holomorphic irreducible conformal net extension $\A_{V_{U_1}}^U$ of 
$\A_{V_{U_1}}$ such that the conformal Hamiltonian $L_0$ of $\A_{V_{U_1}}^U$ on the vacuum Hilbert space $\mathcal{H}_U$ satisfies $\operatorname{dim}( \operatorname{Ker} (L_0-1_{\mathcal{H}_U}) )= \operatorname{dim} (U_1)$. If $\mathcal{B}$ is a holomorphic irreducible conformal net extension of $\A_{V_{U_1}}$ acting on a Hilbert space $\mathcal{H}_{\mathcal{B}}$ 
such that $\operatorname{dim}( \operatorname{Ker} (L_0-1_{\mathcal{H}_{\mathcal{B}}}) )= \operatorname{dim} (U_1)$, then
$\mathcal{B}$ is isomorphic to  $\A_{V_{U_1}}^U$.
\end{theo}

\begin{proof}
Recall from the discussion just preceding the statement of this theorem, that the strongly rational holomorphic VOAs $U$ with central  charge $c=24$ and with $U_1$ being a semisimple Lie algebra are completely classified and there exist exactly 69 such VOAs with corresponding weight-one subspaces as listed by Schellekens in \cite[Table 1]{Sch93}. Note also that the affine vertex subalgebra $V_{U_1}$ generated by $U_1$ is the tensor product of simple unitary affine VOAs, see e.g.\ \cite[Remark 5.7c]{Kac01}.
Consequently, $U$ is (completely) unitary with $V_{U_1}$ as unitary subalgebra by Corollary \ref{cor:extensions_notorious_VOAs}.
Moreover, for every such VOA $U$ there exists an irreducible conformal net $\A_{V_{U_1}}^U$ extending $\A_{V_{U_1}}$, acting on the Hilbert space completion $\mathcal{H}_U$ of $U$ and such that $\operatorname{dim}( \operatorname{Ker} (L_0-1_{\mathcal{H}_U}) )= \operatorname{dim} (U_1)$. 
We also have that $\A_{V_{U_1}}$ is completely rational and $\Repu(U)$ is equivalent to $\Repf(\A_{V_{U_1}}^U)$ as a unitary modular tensor category, so that $\A_{V_{U_1}}^U$ is holomorphic.

Concerning the uniqueness statement, let $\B$ be a holomorphic conformal net extension of $\A_{V_{U_1}}$, acting on a Hilbert space $\mathcal{H}_\B$ and satisfying $\operatorname{dim}( \operatorname{Ker} (L_0-1_{\mathcal{H}_\B}) )= \operatorname{dim} (U_1)$. Let $U_\B$ be the finite energy subspace of $\mathcal{H}_\B$, namely the linear span of the eigenvectors of the conformal Hamiltonian $L_0$ of the net $\B$. Note that $U_\B$ is dense on $\mathcal{H}_\B$ Then, through the *-functor $\mathfrak{F}$ of strong integrability, $U_\B$ is a $V_{U_1}$-module, which integrates to a representation of $\A_{V_{U_1}}$ on $\mathcal{H}_\B$. Then, the normalized haploid (hence standard) commutative Q-system associated to $\B$ in $\Repf(\A_{V_{U_1}})$ gives rise to a structure of simple unitary VOA on $U_\B$ by Corollary \ref{cor:extensions_notorious_VOAs}. Moreover, $U_\B$ is also holomorphic by the fact that 
$\Repu(U_\B)\cong \Repf(\A_{V_{U_1}}^{U_\B}) \cong \Repf(\B)$ by Corollary \ref{cor:extensions_notorious_VOAs}  . The Lie algebra 
$(U_\B)_1= \operatorname{Ker} (L_0-1_{\mathcal{H}_{\mathcal{B}}})$ contains a subalgebra isomorphic to $U_1$. Since by assumption 
 $\operatorname{dim}\left((U_\B)_1\right) =  \operatorname{dim}\left( U_1\right)$, then $(U_\B)_1$ and $U_1$ are isomorphic semisimple Lie algebras. Therefore, $U_\B$ must be isomorphic to $U$ by the classification result. It follows that $\B$ is isomorphic to 
 $\A_{V_{U_1}}^U$.
\end{proof}

We now discuss the Leech lattice case. Indeed, for entry 1 of \cite[Table 1]{Sch93}, i.e., the 24-dimensional abelian Lie algebra, there exists a unique up to isomorphism strongly rational holomorphic VOA with $c=24$, whose weight-one subspace is abelian and 24-dimensional. This is the Leech lattice VOA $V_\Lambda$, arising from the Leech lattice $\Lambda$. Clearly, $(V_\Lambda)_1$ generates a vertex subalgebra isomorphic to the tensor product of 24 copies of the \textbf{Heisenberg VOA} $M(1)$, see e.g.\ \cite[Section 3.5]{Kac01}. Recall that $M(1)$ is unitary, see \cite[Section 4.3]{DL14}, and strongly local, see \cite[Example 8.6]{CKLW18}, and so is the 24-times tensor product with itself $M(1)^{\otimes 24}$, see the penultimate row of Table \ref{table:prop_notorious_VOAs}. The irreducible conformal net associated to $M(1)$ is the well-known \textbf{free Bose chiral field net} $\A_{\Uone}$, see e.g.\ \cite[Section 1B]{BMT88}. Recall also that lattice VOAs are strongly local, see the third row of Table \ref{table:prop_notorious_VOAs}, and thus $V_\Lambda$ is too. Therefore, there already exists a Leech lattice holomorphic conformal net $\A_{V_\Lambda}$, coinciding with the one constructed in \cite[Section 3]{DX06}, see \cite[Theorem 2.7.11]{Gui20}, for which the following holds:

\begin{theo}  \label{theo:leech_uniqueness}
	Let $V_\Lambda$ be the Leech lattice VOA and $\A_{V_\Lambda}$ be the associated irreducible conformal net. If $\B$ is a holomorphic conformal net extension of $\A_{\Uone}^{\otimes 24}$, such that the conformal Hamiltonian $L_0$ of $\B$ on the vacuum Hilbert space $\mathcal{H}_\B$ satisfies $\operatorname{dim}( \operatorname{Ker} (L_0-1_{\mathcal{H}_\B}) )= 24$, then $\B$ is isomorphic to $\A_{V_\Lambda}$. Furthermore, $\Aut(\A_{V_\Lambda})=\Aut_{\scalar}(V_\Lambda)$.
\end{theo}

\begin{proof}
Note that every conformal net extension of $\A_{\Uone}^{\otimes 24}$ is an even 24-dimensional lattice conformal net by \cite{Sta95}, cf.\ the beginning of \cite[Section 2]{KL06} and \cite[Section 3.4, p.\ 838]{Bis12}. Then, we can suppose that $\B$ is isomorphic to the conformal net $\A_L$ for some even 24-dimensional lattice $L$. Thanks to \cite[Proposition 3.15]{DX06}, $\B$ is holomorphic if and only if $L$ is unimodular. Moreover, the fact that $\operatorname{dim}( \operatorname{Ker} (L_0-1_{\mathcal{H}_\B}) )= 24$ implies that $L$ has no roots. Yet, it is known, see e.g.\ \cite{CS99}, that the Leech lattice is the unique even unimodular 24-dimensional lattice with no roots and thus $\B$ must be isomorphic to $\A_{V_\Lambda}$. Finally, $\Aut(\A_{V_\Lambda})=\Aut_{\scalar}(V_\Lambda)$ by \cite[Theorem 6.9]{CKLW18}.
\end{proof}

Theorem \ref{theo:leech_uniqueness} poses the natural question (often hard to settle) on whether the 69 VOAs appearing in Theorem \ref{theo:unitarity_holomorphicVOAs} are \emph{strongly local} too. Recall that the Moonshine VOA is known to be strongly local by \cite[Theorem 8.15]{CKLW18}. It turns out that all of them are indeed strongly local, as we show in the following:

\begin{theo} \label{theo:strong_locality_holomorphic_VOAs}
	Every strongly rational holomorphic VOA $U$ with central  charge $c=24$ and $U_1\not=\{0\}$ is strongly local. If $U_1$ is a semisimple Lie algebra, then $\A_{V_{U_1}}^U$ is isomorphic to $\A_U$ and $\Aut(\A_U)=\Aut_{\scalar}(U)$.
\end{theo}

\begin{proof}
Recall that if $U$ is the Leech lattice VOA $V_\Lambda$, then it is strongly local. Accordingly, we can assume that $U_1$ is a semisimple Lie algebra. We split the proof into several steps. We are going to use the theory developed in \cite{CKLW18} throughout, and thus we refer to it for notation and details. 

\textit{First step:} 
we are going to find a set of generators for $U$ and to fix a suitable unitary structure on it.	
By \cite[Section 6.3]{MS20}, see in particular \cite[Theorem 6.6]{MS20} and its proof, there exists a \textit{generalized deep hole} $g\in\Aut(V_\Lambda)$, see \cite[Section 6.2]{MS20}, such that $U$ is isomorphic to the \textit{orbifold construction} $V_\Lambda^{\mathrm{orb}(g)}$, see \cite[Section 3.3]{MS20} and references therein. In particular, $g$ is of finite order and there exists an automorphism $h\in\Aut(U)$ of the same order as $g$, such that $V_\Lambda$ is isomorphic to the \textit{inverse} (or reverse) orbifold construction $U^{\mathrm{orb}(h)}$. Let $G$ and $H$ be the cyclic groups of order $n$ generated by $g$ and $h$ respectively. 
Consider the VOA inclusions $U^H\subset \widetilde{U}\subseteq U$, where $\widetilde{U}$ is generated by the fixed point subalgebra $U^H$ and by $V_{U_1}$. By \cite[Theorem 1]{DM97}, cf.\ also \cite[Theorem 1]{HMT99}, $\widetilde{U}=U^{\langle h^k\rangle}$ for some $k\in\Zpluseq$, which either is 0 or divides $n$. Here, $\langle h^k\rangle$ denotes the subroup of $H$ generated by $h^k$.
On the one hand, $(U^{\langle h^k\rangle})_1=U_1$ as $U_1\subset \widetilde{U}$. On the other hand, the order of the restriction of $h$ as automorphism of $U_1$ equals $n$, cf.\ the last paragraph of \cite[Section 3.2]{MS21}. Then, it must be $k=n$ and thus $\widetilde{U}=U$. 
By Theorem \ref{theo:unitarity_holomorphicVOAs}, $U$ is unitary.
By Proposition \ref{prop:unitary_structure_compact_aut_subgr}, there exists a unitary structure on $U$ making $U^H$ be a unitary subalgebra of $U$ and $h$ be a unitary automorphism. From now on, we keep fixed such unitary structure. Note also that any PCT operator $\theta$ on $U$ preserves the conformal vector and thus the $L_0$ grading too. Therefore, $V_{U_1}$ is a unitary subalgebra of $U$ for any unitary structure we choose on $U$ by \cite[Proposition 5.23]{CKLW18}. 

\textit{Second step:} 
we are going to prove that $U$ is \textit{energy-bounded} and that $U^H$ is strongly local. 
For the energy boundedness of $U$, note that 
$U$ is a simple unitary CFT type VOA extension of the affine VOA $V_{U_1}$ and that $U_1$ is a semisimple Lie algebra. Then, the result follows directly from \cite[Theorem 4.8]{CT} (in the ADE cases one could also use \cite[Theorem 4.7]{Gui20}).
Now, note that the fixed points subalgebras $V_\Lambda^G$ and $U^H$ are isomorphic. Therefore, by Proposition \ref{prop:unitary_structure_compact_aut_subgr} applied to $V_\Lambda$ with $G$ and by \cite[Proposition 7.6]{CKLW18}, $V_\Lambda^G$ is strongly local and thus $U^H$ is strongly local too as a consequence of 
\cite[Theorem 6.8]{CKLW18}.
As a consequence, we have two different nets of von Neumann algebras associated to $U^H$. On the one hand, we can consider the net $\B_{U^H}$ on $\mathcal{H}_U$ generated by the \textit{energy-bounded smeared vertex operators} of $U^H$ as unitary subalgebra of $U$ as in \cite[Eq.\ (114)]{CKLW18}. On the other hand, we can construct the irreducible conformal net $\A_{U^H}$ on $\mathcal{H}_{U^H}$ by strong locality.

As a preparation for the coming step, consider $\mathcal{H}_{U^H}$ as a subspace of $\mathcal{H}_U$ and denote by $E$ the corresponding orthogonal projection. Let $\J$ be the set of \textit{proper intervals} of the circle $S^1$. Clearly, $E\in \B_{U^H}(I)'$ for all $I\in \J$.

\textit{Third step:} 
we are going to prove that $E\B_{U^H}(I)E=\A_{U^H}(I)$ for all $I\in\J$, which implies that $\{E\B_{U^H}(I)E \mid I\in \J\}$ determines an irreducible conformal net on $\mathcal{H}_{U^H}$ by the locality of $\A_{U^H}$. 
Consider a smeared vertex operator $Y(a,f)$ affiliated to $\B_{U^H}(I)$, where $a\in U^H$ and $\mathrm{supp}f\subset I$ for some fixed $I\in\J$. It is easy to see that $EY(a,f)E=\widehat{Y}(a,f)$, where the $\widehat{\cdot}$ is used for smeared vertex operators affiliated to $\A_{U^H}(I)$. Consider the polar decomposition of $Y(a,f)$ as operator on  $\B_{U^H}(I)$, i.e., $Y(a,f)=V(a,f)T(a,f)$ with $T(a,f):=\abs{Y(a,f)}$. Similarly, $\widehat{Y}(a,f)=\widehat{V}(a,f)\widehat{T}(a,f)$. As $E$ commutes with $Y(a,f)$, it commutes also with $V(a,f)$ and $T(a,f)$. Then, we have that $\widehat{Y}(a,f)=(EV(a,f)E) (ET(a,f)E)$ and thus $\widehat{V}(a,f)=EV(a,f)E$ and $\widehat{T}(a,f)=ET(a,f)E$. In conclusion,
	$$
		\B_{U^H}(I)=\{V(a,f) \,,\, V(a,f)^* \,,\, e^{iT(a,f)} \mid a\in U^H \,,\, \mathrm{supp}f\subset I \}''	
	$$
and thus:
	$$
		\A_{U^H}(I)=\{EV(a,f)E \,,\, EV(a,f)^*E \,,\, Ee^{iT(a,f)}E \mid a\in U^H \,,\, \mathrm{supp}f\subset I \}''=E\B_{U^H}(I)E	\,.
	$$
By the arbitrariness of $I\in\J$, we have that $E\B_{U^H}(I)E=\A_{U^H}(I)$ for all $I\in\J$.

\textit{Fourth step:} 
we are going to prove that $\B_{U^H}(I)\subset \B_{U^H}(I')'$ for all $I\in \J$, i.e., the locality for $\B_{U^H}$. 
Let $I\in\J$. By M{\"o}bius covariance, it is enough to prove that if $I_1, I_2\in \J$ are such that $\overline{I_1}\subset I$ and $\overline{I_2}\subset I'$, then $\B_{U^H}(I_1)\subset \B_{U^H}(I_2)'$. Fix such $I_1,I_2\in\J$ and let $A_1\in \B_{U^H}(I_1)$ and $A_2\in\B_{U^H}(I_2)$. Let us denote by $M$ the subspace of $U$ defined by 
$M := U \cap \operatorname{Ker}([A_1,A_2])$.  From the third step, we have that $EA_iE\in \A_{U^H}(I_i)$ for all $i\in\{1,2\}$ and thus $[A_1,A_2]b=0$ for all $b\in U^H$. Thus, $U^H \subset M$.

Now, let $J_1$ and $J_2$ be two disjoint intervals and let $f,g\in C^\infty(S^1)$ be such that $\mathrm{supp}f\subset J_1$ and $\mathrm{supp}g\subset J_2$. If $a\in U_1$ is \textit{Hermitian}, i.e., $\theta(a)=-a$, then the smeared vertex operator $Y(a,f)$ is self-adjoint, whenever $f$ is real, because it satisfies the \textit{linear energy bounds}, see the proof of \cite[Proposition 6.3]{CKLW18} with \cite[Eq. (113)]{CKLW18} and  \cite[Proposition 2]{Nel72}. Moreover, $Y(a,f)$ strongly commutes with every $Y(b,g)$ whenever $b\in U$ by \cite[Lemma 3.6]{CTW22}. Let $J\in \J$ be disjoint from $I_1$ and $I_2$ and pick $f\in C^\infty(S^1)$ such that its support is contained in $J$. Therefore, $Y(a,f)$ commutes with elements of $\B_{U^H}(I_1)$ and of $\B_{U^H}(I_2)$, so that $[A_1,A_2]Y(a,f)b=Y(a,f)[A_1,A_2]b=0$ for all $b\in M$. Hence, $[A_1,A_2] Y(a,f)b = 0$ for all $b \in M$ and all $a \in U_1$ because 
any $a \in U_1$ is a linear combination of two Hermitian vectors in $U_1$.

Now, let $c\in U$ and let $\mathcal{K}_c$ be the closed subspace generated by vectors of type $Y(a,f)c$ with $a\in U_1$ and $f\in C^\infty(S^1)$ such that $\mathrm{supp}f$ is in $J$. By a Reeh-Schlieder argument as in the proof of \cite[Theorem 8.1]{CKLW18}, cf.\ also \textit{Step 2} in the proof of \cite[Theorem 5.2.1]{Gau21}, we can prove that $\mathcal{K}_c$ contains vectors of type $Y(a,f)c$ with $a\in U_1$ and $f\in C^\infty(S^1)$ (without any restriction on the support of $f$). This implies that $a_{(n)}c\in \mathcal{K}_c$ for all $a\in U_1$ and all $n\in \Z$. Consequently, $a_{(n)}b \in M$ for all $a\in U_1$, all $b \in M$ and all $n\in \Z$. By the fact that $U_1$ generates 
the affine vertex subalgebra $V_{U_1}$ and by the Borcherds formula, see \cite[Section 4.8]{Kac01}, it follows that $M$ is a
$V_{U_1}$-submodule of $U$ containing $U^H$, so that $M$ contains $V_{U_1} \cdot U^H$, the linear span of vectors of the form 
$a_{(n)}b$ with $a \in V_{U_1}$, $b \in U^H$ and $n \in \mathbb{Z}$. Using the skew-symmetry \cite[Section 4.2]{Kac01}, it is not difficult to check that $b_{(n)}a\in V_{U_1}\cdot U^H$ for all $b\in U^H$, all $n\in \Z$ and all $a\in V_{U_1}$. This implies that
$U^H\cdot  V_{U_1} \subset M$.  

Now, we come back  to the orbifold construction of $U$ from $V_\Lambda$. By \cite[Section 4, p.\ 11]{MS21}, $U$ is constructed as a direct sum of irreducible $U^H$-modules, i.e., $U = V_\Lambda^{\mathrm{orb}(g)}=\bigoplus_{i=0}^{n-1} V_\Lambda^{(i,0)}$, where $U^H=V_\Lambda^{(0,0)}$ and every $V_\Lambda^{(i,0)}$ is a simple current for $U^H$. Moreover, $h$ acts on every $V_\Lambda^{(i,0)}$
by a scalar $\chi_i \in \mathbb{C}$ and $\chi_i \neq \chi_j$ for $i \neq j$.

The VOA automorphisms of  $U$ leave $U_1$, and hence $V_{U_1}$, globally invariant. It follows that the action of $H$ on 
$U$ restricts to an action on $V_{U_1}$. The latter action is faithful because the order of $h\restriction_{U_1}$ equals the order of $h$, 
see again the last paragraph of \cite[Section 3.2]{MS21}. By \cite[Theorem 2]{DM97},  $V_\Lambda^{(i,0)} \cap V_{U_1} \neq \{ 0 \}$ for all $i$ and thus  $V_\Lambda^{(i,0)}  \subset U^H\cdot  V_{U_1}$ fo all $i$ so that $U \subset U^H\cdot  V_{U_1} \subset M$.  
It follows that $[A_1, A_2]=0$, proving the locality for $\B_{U^H}$.

\textit{Fifth step:} we are ready to move to the conclusion of the proof. 
By the fourth step together with \cite[Theorem 8.1]{CKLW18} and \cite[Lemma 3.6]{CTW22}, $\widetilde{U}$ is strongly local. Therefore, $U$ is strongly local as we have proved that $\widetilde{U}=U$ in the first step. 
Finally, by the uniqueness result of Theorem \ref{theo:unitarity_holomorphicVOAs}, $\A_{V_{U_1}}^U$ is isomorphic to $\A_U$ and thus $\Aut(\A_U)=\Aut_{\scalar}(U)$ by \cite[Theorem 6.9]{CKLW18}.
\end{proof}

\begin{rem}\label{rem: Aut_V Schellekens} If $V$ is a holomorphic VOA with $c=24$ and $V_1 \neq \{ 0\}$ then $\Aut(V)$ is described in \cite{BLS22}. Hence, by Theorem \ref{theo:strong_locality_holomorphic_VOAs} and Theorem  \ref{theo: unitary maximally compact}, the isomorphism class of $\Aut(\A_V)$ can be described  by any maximal compact subgroup of $\Aut(V)$.    
\end{rem}

We end this section with two remarks on the relation between the holomorphic conformal nets given by Theorem \ref{theo:unitarity_holomorphicVOAs} and the various constructions of some of them present in the literature.  The main point is that the strong locality proved in Theorem \ref{theo:strong_locality_holomorphic_VOAs} gives a uniform way to construct these models. 

\begin{rem}  \label{rem:other_holomorphic_conformal_net_constructions}
In \cite{LS15}, see in particular \cite[Remark 1.2]{LS15} for the corresponding entries in the Schellekens list, the holomorphic \emph{framed} VOAs with central charge $c=24$ were classified. There exist exactly 56 of such VOAs up to isomorphism. One of them is the Moonshine VOA \cite{FLM88}, see also \cite{DGH98} and \cite{Miy04}. By tensor categorical methods, corresponding holomorphic framed conformal nets with $c=24$ were constructed in \cite{KL06} and in \cite{KS14}. 
Actually, they must correspond to the holomorphic framed conformal nets, arising from the strong locality of the holomorphic framed VOAs with $c=24$, given by Theorem \ref{theo:strong_locality_holomorphic_VOAs} and \cite[Theorem 8.15]{CKLW18}.  Indeed, by Corollary \ref{cor:framed_chiral_CFTs} and  \cite{LS15}, there are exactly 56 isomorphism classes of framed holomorphic conformal nets with $c=24$.  By Theorem  \ref{theo:strong_locality_holomorphic_VOAs} and  \cite[Theorem 8.15]{CKLW18}, each framed holomorphic VOA with $c=24$ $V$ is strongly local. By Theorem \ref{theo:unitarity_holomorphicVOAs}, Theorem \ref{theo:strong_locality_holomorphic_VOAs},  \cite[Theorem 8.15]{CKLW18} and \cite[Theorem 3.6]{KL06}, each $\A_V$ is a holomorphic conformal net with $c=24$ which is framed by \cite[Theorem 7.5]{CKLW18}.
Moreover, it follows from \cite[Theorem 9.2]{CKLW18} that if $V$ and $W$ are simple strongly local unitary VOAs then $V$ and $W$ are isomorphic if and only if  $\A_V$ and $\A_W$ are isomorphic. Hence, if $\B$ is a framed holomorphic conformal net with $c=24$ then $\B$ is isomorphic to $\A_V$ for some framed holomorphic VOA $V$ with $c=24$. 
Note that Corollary \ref{cor:framed_chiral_CFTs} also gives back  the construction of the Moonshine conformal net given in \cite{KL06}.
\end{rem}

\begin{rem} \label{rem:other_holomorphic_conformal_net_constructions2}
Other models were constructed by the technique of mirror extension in \cite[Section 3.2]{Xu09} (cf.\ \cite{LL20} for a parallel VOA application of mirror extension). There, the author gives three holomorphic irreducible conformal net extensions of affine conformal nets, corresponding to entries 18, 27 and 40 of the Schellekens list, see \cite[Theorem 3.3]{Xu09}. By the uniqueness part of Theorem \ref{theo:unitarity_holomorphicVOAs}, these models must be isomorphic to the ones obtained using strong locality by Theorem \ref{theo:strong_locality_holomorphic_VOAs}. In particular, note that nos.\ 27 and 40 are framed models.   Moreover,  in \cite[Section 4]{Xu18}, the author gives three holomorphic conformal net extensions of the conformal nets associated to the affine VOAs corresponding to entries \ 9, 11 and 20 of the Schellekens list. By the uniqueness part of Theorem \ref{theo:unitarity_holomorphicVOAs}, these three models must be isomorphic to the ones obtained via strong locality by Theorem \ref{theo:strong_locality_holomorphic_VOAs}.  Actually, one could probably use the construction in \cite[Section 4]{Xu18} together with the generalized deep hole construction of strongly rational 
holomorphic VOAs with $c=24$ and non-zero weight-one subspace to obtain all the conformal nets obtained through Theorem \ref{theo:strong_locality_holomorphic_VOAs}, 
by a conformal net analogue of the orbifold construction from the Leech lattice conformal net constructed in \cite{DX06} and in \cite{KL06}. 
\end{rem}

\section{Classification of superconformal VOSAs}   \label{sec:superconformal_VOAs}

In this section, we discuss the classification of $N=0$, $N=1$ and $N=2$ superconformal VOSAs, by transferring the already known classification in the superconformal net setting, using the tools developed in Section \ref{sec:applications}.

First, we give an alternative proof of an already known classification result for $N=0$ superconformal VOAs, see \cite[Section 4]{DL15} and \cite[Corollary 2.22]{Gui21c}, namely the classification of simple CFT type extensions of the unitary Virasoro VOAs with central  charge $c<1$, see the first row of Table \ref{table:prop_notorious_VOAs}. 

\begin{theo}
	Every simple CFT type VOA  extension $U$ of a unitary Virasoro VOA with central  charge $c<1$ is unitary. Furthermore, these VOA extensions are in one-to-one correspondence with the irreducible conformal nets of the same central charge $c$, which are completely classified in \cite[Section 5]{KL04}.
\end{theo}

\begin{proof}
The claim follows directly from Corollary \ref{cor:extensions_notorious_VOAs}.
\end{proof}

In a similar way, we have two classification results \cite[Section 7]{CKL08} and \cite[Section 6]{CHKLX15} for the $N=1$ and $N=2$ superconformal nets respectively. By definition, see e.g.\ \cite[Section 5.9]{Kac01}, the $N=1$ and $N=2$ superconformal VOSAs are extensions of the $N=1$ and $N=2$ \textbf{super-Virasoro VOSAs} respectively and we will consider the case where the corresponding super-Virasoro subalgebra are unitary with $c < 3/2$ and $c < 3$ respectively (the unitary discrete series).
Then, the crucial point is that the even part of these super-Virasoro VOSAs can be realized as regular cosets of certain VOAs in Table \ref{table:prop_notorious_VOAs}. Thanks to this, we are allowed to apply Corollary \ref{cor:extensions_coset} and get the desired classifications via Theorem \ref{theo:equiv_VOSA_net_super_extensions}.

\begin{theo} \label{theo:classification_N=1}
	Every simple CFT type VOSA extension $U$ of a unitary super-Virasoro VOA with central charge $c<\frac{3}{2}$  is a unitary VOSA. Furthermore, these $N=1$ superconformal VOSAs are in one-to-one correspondence with the irreducible superconformal nets of the same central charge $c$, which are completely classified in \cite[Section 7]{CKL08}.
\end{theo}

\begin{proof}	
Fix $n\in\Zplus$ such that the central charge $c$ equals $c_n<\frac{3}{2}$ as in \cite{GKO98}. Let $V_\parzero$ be the even part of the $N=1$ super-Virasoro VOSA $V:=V^{c_n}(NS)$ with central  charge $c_n$. Note that $V_\parzero$ is simple and of CFT type as $V$ is, see Remark \ref{rem:simplicity_even_part}.
It follows from \cite[Section 3 and Section 4]{GKO98}, cf.\  \cite[pp.\ 1103–1104]{CKL08}, that $V_\parzero$ can be realized as the coset VOA of the VOA inclusion $V^{n+2}(\mathfrak{sl}(2,\C))\subset V^n(\mathfrak{sl}(2,\C))\otimes V^2(\mathfrak{sl}(2,\C))$. 
Moreover, $V_\parzero$ is regular, see \cite[Theorem 6.2 and Theorem 6.3]{CFL20}. It is also strongly local, see the last row of Table \ref{table:prop_notorious_VOAs}, and completely unitary by \cite[Theorem I]{Gui20}. Hence, $U$ has a unitary VOSA structure by Corollary \ref{cor:unitarity_vosa_extensions}.

Now, consider the \textit{$N=1$ super-Virasoro net} $\A:=\mathrm{SVir}_{c_n}$ with central  charge $c_n$, as constructed in \cite[Section 6.3]{CKL08}. It is not difficult to check that $\A_\parzero$ is isomorphic to $\A_{V_\parzero}$.  By Corollary \ref{cor:extensions_coset}, $\A_\parzero$ is completely rational and the strong integrability functor $\mathfrak{F}: \Repu(V_\parzero) \to \Repf(\A_\parzero)$ is an equivalence  of unitary modular tensor categories. Moreover, $\mathfrak{F}$ maps the equivalence class of the $\Z_2$-simple current defining $V$ in $\Repu(V_\parzero)$ into the equivalence class of the one defining $\A$ in $\Repf(\A_\parzero)$.
Then, the result follows from Theorem \ref{theo:equiv_VOSA_net_super_extensions}.
\end{proof}

In a similar fashion, we have:

\begin{theo}  \label{theo:classification_N=2}
	Every simple CFT type VOSA extension $U$ of a unitary $N=2$ super-Virasoro VOSA with central charge $c<3$  is a unitary VOSA. Furthermore, these $N=2$ superconformal VOSAs are in one-to-one correspondence with the irreducible superconformal nets of the same central charge $c$, which are completely classified in \cite[Section 6]{CHKLX15}.
\end{theo}

\begin{proof}
Fix $n\in\Zplus$ such that the central charge $c$ equals $c_n<3$ as in \cite{CHKLX15}. Let $V_\parzero$ be the even part of the $N=2$ super-Virasoro VOSA $V:=V^{c_n}(N2)$ with central  charge $c_n$. Note that $V_\parzero$ is simple and of CFT type as $V$ is, see Remark \ref{rem:simplicity_even_part}.
By \cite[Section 5]{CHKLX15}, cf. also \cite[Corollary 8.8]{CL19}, $V$ can be realized as the coset VOSA of the VOSA inclusion $V_{L_{2(n+2)}}\subset V^n(\mathfrak{sl}(2,\C))\otimes F\otimes F$, where $F$ is the real free fermion VOSA, see e.g.\ \cite[Section 3.6 and Proposition 4.10(b)]{Kac01}.
Therefore, $V_\parzero$ is equal to the coset VOA of the VOA inclusion $V_{L_{2(n+2)}}\subset V^n(\mathfrak{sl}(2,\C))\otimes V_{L_4}$, see also \cite[p.\ 1312]{CHKLX15} and \cite[Section 4.4.2]{CKM21}. (Note that the even part of $F\otimes F$ is the even rank-one lattice VOA $V_{L_4}$.) Furthermore, $V$ is regular by \cite[Theorem 7.2]{Ada04}, cf.\ the discussion after \cite[Corollary 8.8]{CL19}. It follows that $V_\parzero$ is regular too, see the claim in \cite[Section 4, p.\ 7789]{DNR21} based on \cite{Miy15} and \cite{CM18}. It is also strongly local, see the last row of Table \ref{table:prop_notorious_VOAs}, and completely unitary by \cite[Theorem I]{Gui20}. Hence, $U$ has a unitary VOSA structure by Corollary \ref{cor:unitarity_vosa_extensions}. 

Now, consider the \textit{$N=2$ super-Virasoro net} $\A:=\mathrm{SVir2}_{c_n}$ with central  charge $c_n$, as constructed in \cite[Section 3]{CHKLX15}. It is not difficult to check that $\A_\parzero$ is isomorphic to $\A_{V_\parzero}$. 
By Corollary \ref{cor:extensions_coset}, $\A_\parzero$ is completely rational and the strong integrability functor $\mathfrak{F}: \Repu(V_\parzero) \to \Repf(\A_\parzero)$ is an equivalence  of unitary modular tensor categories. Moreover, $\mathfrak{F}$ maps the equivalence class of the $\Z_2$-simple current defining $V$ in $\Repu(V_\parzero)$ into the equivalence class of the one defining $\A$ in $\Repf(\A_\parzero)$. Then, the result follows from Theorem \ref{theo:equiv_VOSA_net_super_extensions}.
\end{proof}

\appendix
\section{On the automorphism groups of unitary VOAs}
Let $V = \bigoplus_{n \in \mathbb{Z}}V_n $ be a VOA and let $\Aut(V)$ be its automorphism group. If $g \in \Aut(V)$, then for any $n \in \mathbb{Z}$,  $gV_n \subset V_n$ and the restriction 
$g\restriction_{V_n}$ belongs to $\mathrm{GL}(V_n)$. The map $g \mapsto \left( g\restriction_{V_n} \right)_{n \in \mathbb{Z}}$ gives an isomorphism of 
$\Aut(V)$ onto a closed subgroup of the direct product $\prod_{n \in \mathbb{Z}}\mathrm{GL}(V_n)$ and this makes $\Aut(V)$ into a metrizable topological group, cf.\
\cite[Section 4.3]{CKLW18}.

It is well-known that if $V$ is finitely generated, then $\Aut(V)$ is a finite dimensional Lie group. This can be seen as follows. If $V$ is finitely generated, there is a $N \in \mathbb{Z}_{\geq 0}$ such that $V$ is generated by the finite dimensional vector space $V_{\leq N} := \bigoplus_{n \in \mathbb{Z},\, n\leq N} V_n$ and the map $g \mapsto g\restriction_{V_{\leq N}}$ is a topological isomorphism of $\Aut(V)$ onto a closed subgroup of $\mathrm{GL}(V_{\leq N})$. It follows that   if $V$ is finitely generated then $\Aut(V)$ is a finite dimensional Lie group. 
Actually, if $V$ is finitely generated, the image of the restriction map  $g \mapsto g\restriction_{V_{\leq N}}$ is Zariski closed in $\mathrm{GL}(V_{\leq N})$ 
and hence $\Aut(V)$ is isomorphic to a complex linear algebraic group, cf.\ \cite[Theorem 2.4]{DG02}.  As a consequence, if $\Aut(V)_0$ denotes the connected component of the identity in $\Aut(V)$, then the quotient group $\Aut(V)/\Aut(V)_0$ is finite, cf.\ \cite[Section 1]{DG02}.  Note that every strongly rational VOA is finitely generated by \cite{GN03} or by \cite{DZ08}.

As usual, if $V$ is a VOA and $H$ is a subgroup of  $\Aut(V)$, we will denote by $V^H$ the fixed point subalgebra of $V$, namely 
$V^H=\{a \in V \mid \forall g \ (g\in H \Rightarrow  ga=a ) \}$.  

From now on, $V$ will denote a simple unitary VOA with invariant normalized scalar product $(\cdot |\cdot)$ and PCT operator $\theta$, see \cite[Chapter 5]{CKLW18}.
The unitary automorphism group $\Aut_{(\cdot |\cdot)}(V)\subset \Aut(V)$ is defined by 
\begin{equation} \label{eq:unitaryAut}
\Aut_{(\cdot |\cdot)}(V) := \{ g  \in \Aut(V) \mid  \forall a (a \in V \Rightarrow (ga|ga) = (a|a) ) \} .
\end{equation}
It is a compact subgroup of $\Aut(V)$, see \cite[Section 5.3]{CKLW18}. 

We say that $g \in \Aut(V)$ is {\bf strictly positive} if $(a|ga) > 0$ for all $a \in V\setminus \{ 0 \}$ and denote by 
$\Aut_{+}(V)$ the set of strictly positive elements of $\Aut(V)$. We have that $1_V \in \Aut_{+}(V)$ and we say that   
$\Aut_{+}(V)$ is trivial if $ \Aut_{+}(V)  = \{1_V \}$. If $g \in \Aut_{+}(V)$  and $\alpha \in \mathbb{C}$, then 
$g^{\alpha}$ is a well defined element of $\Aut(V)$ and the map $\mathbb{C} \ni \alpha \mapsto g^\alpha = e^{\alpha \log g } \in \Aut(V)$ is a continuous group homomorphism, cf.\ \cite[Section 5.3]{CKLW18}. It is not difficult to show that, for any $\alpha \in \mathbb{C}$, the map  $\Aut_{+}(V) \ni g \mapsto g^\alpha \in \Aut(V)$ is continuous. Moreover,  $g^t \in \Aut_{+}(V)$ for all $t \in \mathbb{R}$. 
 It follows that $\Aut_{+}(V)$ is a path connected subset of $\Aut(V)$ containing $1_V$ and hence $\Aut_{+}(V) \subset \Aut(V)_{0}$. 
Note that $\Aut_{+}(V)$ is not in general a subgroup of  $\Aut(V)$. However, $g^{-1} \in \Aut_{+}(V)$ for all $g \in \Aut_{+}(V)$ 
and, if $g$ and  $h$ are commuting elements of  $\Aut_{+}(V)$ then $gh= h^{\frac12}gh^{\frac12} \in \Aut_{+}(V)$. 
It follows that if  $\Aut(V)_{0}$ is abelian, then  $\Aut_{+}(V)$ is an abelian closed and connected subgroup of $\Aut(V)$. It is clear that if 
$\Aut(V)_{0}$ is trivial, then $\Aut_{+}(V)$  is also trivial. The converse is also true as a consequence of \cite[Theorem 5.21]{CKLW18}.

For $g \in \Aut(V)$ we set $g^*= \theta g^{-1} \theta \in \Aut(V)$. Then, $(a|gb) = (g^*a|b)$ for all $a,b \in V$. In particular, 
$g^*g \in   \Aut_{+}(V)$ for all $g \in \Aut(V)$ and an automorphism $g \in \Aut(V)$ is unitary, i.e.\ it belongs to  $\Aut_{(\cdot |\cdot)}(V)$, if and only if $g^* = g^{-1}$.  The next proposition shows that the elements of $\Aut(V)$ admit a polar decomposition. 

\begin{prop} \label{prop:polar} Let $V$ be a simple unitary VOA and let $g \in  \Aut(V)$. Then, there is a unique $|g| \in   \Aut_{+}(V)$ and a unique $u \in \Aut_{(\cdot |\cdot)}(V)$ such that $g = u |g|$.   
\end{prop}
\begin{proof} Set $|g| := (g^*g)^{\frac12}$ and $u := g |g|^{-1}$. Then,  it is straightforward to see that  $|g| \in   \Aut_{+}(V)$ and $u \in \Aut_{(\cdot |\cdot)}(V)$. In order to prove the uniqueness part let us assume that $g=vh$ with $v$ a unitary automorphism and 
$h$ a strictly positive automorphism. Then, 
$g^*g = hv^*vh = h^2$ so that $h = (g^*g)^{\frac12} = |g|$ and hence $v=gh^{-1} = u$.  
 \end{proof}
 
We have the following consequence: 
 
 \begin{prop}  \label{prop: quotient Aut} 
 Let $V$ be a simple unitary VOA  and let $q:\Aut(V) \to \Aut(V)/\Aut(V)_0$ be the quotient map. 
 Then $q\restriction_{\Aut_{(\cdot |\cdot)}(V)}: \Aut_{(\cdot |\cdot)}(V) \to  \Aut(V)/\Aut(V)_0$ is a surjective group homomorphism. 
 In particular $\Aut(V)$ is almost connected, i.e.\  $\Aut(V)/\Aut(V)_0$ is compact. Moreover, if $\Aut(V)$ is a finite dimensional Lie group, 
 then $\Aut(V)/\Aut(V)_0$ is a finite group and the map $q\restriction_{\Aut_{(\cdot |\cdot)}(V)}: \Aut_{(\cdot |\cdot)}(V) \to  \Aut(V)/\Aut(V)_0$ factors through an isomorphism of   $\Aut_{(\cdot |\cdot)}(V) / \Aut_{(\cdot |\cdot)}(V)_0$ onto  $\Aut(V)/\Aut(V)_0$.  
   \end{prop}
   \begin{proof} Let $g \in \Aut(V)$. By Proposition \ref{prop:polar} $g= u|g|$ with $|g| \in   \Aut_{+}(V)$ and $u \in \Aut_{(\cdot |\cdot)}(V)$.
   Hence, since $\Aut_+(V) \subset \Aut(V)_0$, $q(g) = q(u) \in q (  \Aut_{(\cdot |\cdot)}(V) )$ and hence $q\restriction_{\Aut_{(\cdot |\cdot)}(V)}$ is surjective. 
   Now, assume that $\Aut(V)$ is a finite dimensional Lie group. Then, $\Aut(V)_0$ is a path connected open subgroup of 
   $\Aut(V)$. Hence, $\Aut(V)/\Aut(V)_0$, being discrete and compact must be finite. Moreover, since  
   $\Aut_{(\cdot |\cdot)}(V)_0 \subset \Aut(V)_0$,  the map $q\restriction_{\Aut_{(\cdot |\cdot)}(V)}: \Aut_{(\cdot |\cdot)}(V) \to  \Aut(V)/\Aut(V)_0$ factors through an homomorphism of   $\Aut_{(\cdot |\cdot)}(V) / \Aut_{(\cdot |\cdot)}(V)_0$ onto  $\Aut(V)/\Aut(V)_0$. In order to prove that the latter homomorphism is injective we have to show that  $\Aut(V)_0 \cap  \Aut_{(\cdot |\cdot)}(V) \subset 
   \Aut_{(\cdot |\cdot)}(V)_0$.  Let $u \in  \Aut(V)_0 \cap  \Aut_{(\cdot |\cdot)}(V)$. Since $\Aut(V)_0$ is path connected, there is a continuous path $[0,1] \ni t \mapsto g(t) \in \Aut(V)$ with $g(0) = 1_V$ and $g(1) = u$. Then, $[0,1] \ni t \mapsto u(t) := g(t)|g(t)|^{-1} 
   \in \Aut_{(\cdot |\cdot)}(V)$ is a continuous path with $u(0) = 1_V$ and $u(1) = u$. Hence, $u \in   \Aut_{(\cdot |\cdot)}(V)_0$.       
    \end{proof}

  \begin{prop} \label{prop: unitary fixed point} Let $V$ be a simple unitary VOA. Then, 
  $V^{\Aut_{(\cdot |\cdot)}(V)} = V^{\Aut(V)}$.  \end{prop} 
  \begin{proof} Clearly $V^{\Aut(V)} \subset V^{\Aut_{(\cdot |\cdot)}(V)}$. Now, let $a \in   V^{\Aut_{(\cdot |\cdot)}(V)}$ and let 
  $g= u|g|$ be any element of $\Aut(V)$. Since $|g|^{it}$ is unitary for all $t\in \mathbb{R}$, we have  $|g|^{it}a =a$ for all 
  $t\in \mathbb{R}$. 
  By taking the derivative at $t=0$ of the latter equality we get $\log |g| \ a = 0$ and hence $|g| a = e^{\log |g|} a = a$, so $ga=ua=a$ and 
  $a \in V^{\Aut(V)}$.      
  \end{proof} 

\begin{theo} \label{theo: unitary maximally compact} Let $V$ be a simple unitary VOA and assume that $\Aut(V)$ is a finite dimensional Lie group. Then,  $\Aut_{(\cdot |\cdot)}(V)$ is a maximal compact subgroup of $\Aut(V)$. 
\end{theo}
\begin{proof} Let $G$ be a compact subrgoup of $\Aut(V)$ containing  $\Aut_{(\cdot |\cdot)}(V)$. Then, by Proposition 
\ref{prop: unitary fixed point}, $V^G = V^{\Aut_{(\cdot |\cdot)}(V)}$. Moreover, $G$ is a compact Lie group being a compact subgroup of the finite dimensional Lie group $\Aut(V)$. Hence, $G =  \Aut_{(\cdot |\cdot)}(V)$ by the Galois correspondence in \cite[Theorem 3]{DM99}.
\end{proof}

\begin{prop}  \label{prop:unitary_structure_compact_aut_subgr}
	Let $V$ be a simple unitary VOA and assume $\Aut(V)$ is a finite dimensional Lie group. Then, if  $G$ is a compact subgroup of $\Aut(V)$, then there exists an invariant normalized scalar product $\{\cdot|\cdot\}$ on $V$ such that $G \subset \Aut_{\{\cdot |\cdot \}}(V)$, namely the automorphisms in $G$ are unitary with respect to $\{\cdot|\cdot\}$.
\end{prop}  
\begin{proof} 
By Theorem \ref{theo: unitary maximally compact}  $\Aut_{(\cdot |\cdot)}(V)$ is a maximal compact subgroup of $\Aut(V)$. Moreover, 
$\Aut(V)/\Aut(V)_0$ is finite, cf.\ Proposition \ref{prop: quotient Aut}.
Hence, by \cite[Theorem 32.5]{Str06}, see also \cite[XV Theorem 3.1]{Hoc65} and \cite[Section 1]{HT94}, $G$ is contained in some maximal subgroup of $\Aut(V)$ which must be conjugate to $\Aut_{(\cdot |\cdot)}(V)$. As a consequence,
there is an element $h \in \Aut(V)$ such that $hGh^{-1} \subset \Aut_{(\cdot |\cdot)}(V)$. Let  $\{\cdot|\cdot\}$ be defined by  $\{a|b\} :=  (ha|hb)$ for all $a, b \in V$. 
Then, $\{\cdot|\cdot\}$ is a normalized invariant scalar product on $V$ with PCT operator $\tilde{\theta} := h^{-1}\theta h$. If 
$g \in G$ then $hgh^{-1} \in \Aut_{(\cdot |\cdot)}(V)$ so that $\{ga|gb\} = (hga|hgb)= (ha|hb) = \{a|b\}$ for all $a,b \in V$ and hence 
$g \in \Aut_{\{\cdot |\cdot \}}(V)$.
\end{proof}

\bigskip
\noindent
{\small
{\bf Acknowledgements.}
We would like to thank Giulio Codogni, Bin Gui, Andr\'e Henriques, Corey Jones,  Ching Hung Lam, Dave Penneys, Nils Scheithauer and Wei Yuan for stimulating discussions. We also would like to thank the anonymous referee for useful comments and suggestions. S.C.\ is supported in part by GNAMPA-INDAM. L.G.\ is supported by the European Union's Horizon 2020 research and innovation programme H2020-MSCA-IF-2017 under Grant Agreement 795151 \emph{Beyond Rationality in Algebraic CFT: mathematical structures and models}. S.C.\, T.G.\ and L.G.\ acknowledge support from the \emph{MIUR Excellence Department Project} awarded to the Department of Mathematics, University of Rome ``Tor Vergata'', CUP E83C18000100006, and from the University of Rome ``Tor Vergata'' funding \emph{OAQM}, CUP E83C22001800005. S.C.\ and L.G.\ also acknowledge support from the \emph{MIUR Excellence Department Project MatMod@TOV} awarded to the Department of Mathematics, University of Rome ``Tor Vergata'', CUP E83C23000330006.}  Most of this work was done while T.G. was a postdoctoral fellow at the Department of Mathematics of the University of Rome ``Tor Vergata" within the project ``Algebre di operatori con applicazioni alla teoria di campi quantistici".

\bigskip
\noindent
{\small
\emph{Data sharing is not applicable to this article as no new data were created or analyzed in this study.} }


\end{document}